\documentclass[letterpaper, 10pt, journal, onecolumn, final]{IEEEtran}

\IEEEoverridecommandlockouts                              %

\usepackage{cite}
\usepackage{amsmath,amssymb,amsfonts}
\usepackage{graphicx}
\usepackage{algorithm,algorithmicx,algpseudocode}
\usepackage{tcolorbox}
\usepackage{booktabs} %
\usepackage{subcaption}
\usepackage{caption}
\usepackage{soul}

\title{Suboptimal and Reduced-Order MPC\\via Timescale Separation}

\author{Stefano Di Gregorio, Guido Carnevale, Giuseppe Notarstefano%
\thanks{Work partially funded by FISA-2023-00210 project APACHE - CUP J53C25000520001, by the European Union - Next Generation EU - under the National Recovery and Resilience Plan (NRRP), Mission 4, Component 2, Investment 3.3. CUP J33C24001490009 and by IMA S.p.A. 
The corresponding author is S. Di Gregorio.}
\thanks{The authors are with the Department of Electrical, Electronic and Information Engineering, University of Bologna, 40136, Bologna, Italy (e-mail: \{stefano.digregorio, guido.carnevale, giuseppe.notarstefano\}@unibo.it).}%
}%

\newcommand{\du}{^}
\newcommand{\ud}{_}

\newcommand{\oprocend}{\hfill\hbox{\rule{1.25ex}{1.25ex}}}

\newcommand{\T}{^\top}
\newcommand{\map}[3]{#1: #2 \rightarrow #3}

\newtheorem{remark}{Remark}

\newtheorem{assumption}{Assumption}

\newtheorem{lemma}{Lemma}
\newtheorem{theorem}{Theorem}

\newtheorem{proof}{Proof}

\newcommand{\xcont}{x\ct}
\newcommand{\xcontdot}{\dot{x}\ct}

\newcommand{\xicont}{\xi\ct}
\newcommand{\xicontdot}{\dot{\xi}\ct}
\newcommand{\ucont}{u\ct}

\newcommand{\cA}{\mathcal{A}}

\newcommand{\cG}{\mathcal{G}}

\newcommand{\cL}{\mathcal{L}}

\newcommand{\cQ}{\mathcal{Q}}
\newcommand{\cR}{\mathcal{R}}

\newcommand{\cU}{\mathcal{U}}

\newcommand{\cX}{\mathcal{X}}

\newcommand{\R}{\mathbb{R}}
\newcommand{\N}{\mathbb{N}}

\newcommand{\norm}[1]{\left\| #1 \right\|}
\newcommand{\snorm}[1]{\| #1 \|}

\newcommand{\iter}{{t}}
\newcommand{\iterp}{{\iter+1}}

\newcommand{\itermpc}{k}

\newcommand{\tti}{\tau}

\newcommand{\x}{x}
\newcommand{\z}{z}
\newcommand{\uu}{u}

\newcommand{\pr}{\delta}
\newcommand{\prm}{\pr_{\text{\tiny MPC}}}
\newcommand{\prct}{\epsilon}

\newcommand{\xtp}{\x\ud{\iterp}}
\newcommand{\xt}{\x\ud{\iter}}

\newcommand{\xitp}{\xi\ud{\iterp}}
\newcommand{\xit}{\xi\ud{\iter}}

\newcommand{\ztp}{\z\ud{\iterp}}
\newcommand{\zt}{\z\ud{\iter}}

\newcommand{\slow}{f}
\newcommand{\fastsymbol}{g}
\newcommand{\fast}{\fastsymbol\du{\prct}}
\newcommand{\mpc}{\mathcal{A}}

\newcommand{\slowc}{\slow\ct} 
\newcommand{\fastc}{\fastsymbol\ct}

\newcommand{\ct}{\ud{\text{\tiny c}}}

\newcommand{\Fast}{G}

\newcommand{\ut}{\uu\ud{\iter}}

\newcommand{\cost}{\ell}
\newcommand{\costf}{\ell_{\hor}}

\newcommand{\hor}{T}

\newcommand{\red}{\slow\du{\text{\tiny R}}}
\newcommand{\redc}{\red\ct}

\newcommand{\xieq}{\xi\ud{\text{\tiny eq}}}

\newcommand{\zstar}{\z\ud{\star}}

\newcommand{\scale}{.975}

\newcommand{\lip}{L}
\newcommand{\lipeq}{L_{\xi}}

\newcommand{\txi}{\tilde{\xi}}
\newcommand{\txit}{\txi\ud\iter}
\newcommand{\txitp}{\txi\ud\iterp}

\newcommand{\tz}{\tilde{z}}
\newcommand{\tzt}{\tz\ud\iter}
\newcommand{\tztp}{\tz\ud\iterp}

\DeclareMathOperator{\col}{\mathrm{col}}
\DeclareMathOperator{\diag}{\mathrm{diag}}

\newcommand{\tfast}{\tilde{\fastsymbol}\du{\prct}}
\newcommand{\tmpc}{\tilde{\mpc}}

\newcommand{\w}{w}
\newcommand{\wt}{\w\ud\iter}

\newcommand{\y}{y}
\newcommand{\yt}{\y\ud\iter}
\newcommand{\ytp}{\y\ud\iterp}

\newcommand{\ty}{\tilde{\y}}
\newcommand{\tyt}{\ty\ud\iter}
\newcommand{\tytp}{\ty\ud\iterp}

\newcommand{\xstar}{\x\ud{\star}}
\newcommand{\ustar}{\uu\ud{\star}}

\newcommand{\psit}{\psi\ud\iter}
\newcommand{\psitp}{\psi\ud\iterp}

\newcommand{\eq}{h}

\newcommand{\gammam}{\gamma_{\text{max}}}

\newcommand{\feas}{\cX_{\hor}}

\newcommand{\domX}{\mathbf{X}}
\newcommand{\domXi}{\boldsymbol{\Xi}}
\newcommand{\domU}{\mathbf{U}}
\newcommand{\domZ}{\mathbf{Z}}
\newcommand{\domY}{\mathbf{Y}}

\newcommand{\dimz}{n_z}

\newcommand{\tint}{s}

\newcommand{\dt}{\sigma}

\newcommand{\redd}{F_{\text{\tiny MPC}}^{\text{\tiny R}}}

\newcommand{\ratio}{r}

\newcommand{\flow}{\phi}

\newcommand{\bprsi}{\bar{\pr}_1}

\def\algo/{{SMART-MPC}}
\def\algoCAPS/{{SuboptiMAl and Reduced-order Tunable MPC}}

\begin{document}

\maketitle
\thispagestyle{empty}
\pagestyle{empty}

\begin{abstract}     
        In this paper, we propose a generalized framework for the
        design and analysis of suboptimal and
        reduced-order nonlinear Model Predictive Control (MPC) architectures.
        The proposed framework manages real-time operation of MPC
        schemes by (i) computing the control action suboptimally, i.e., by
        running a generic optimal
        control algorithm for a finite number of iterations, and (ii) relying on a reduced-order model
        that neglects part of the plant dynamics (accounting for, e.g.,
        unmodeled dynamics or a low-level compensator).
        To rigorously handle the interplay between optimization error
        and model mismatch, we treat the sampling time as a tunable design parameter. 
        We analyze the resulting closed-loop system,
        comprising the full-order physical plant interconnected with
        the iterative optimization algorithm (treated as a dynamical
        system), by leveraging tools from timescale separation.
        We prove that
        operating at a sufficiently fast sampling rate ensures that the closed-loop
        system maintains recursive feasibility and achieves an exponentially stable equilibrium point.
        The effectiveness of the proposed framework is validated on an underactuated two-link robotic arm through virtual experiments in the high-fidelity MuJoCo physics engine.
    \end{abstract}

\section{Introduction}
\label{sec:introduction}

Model Predictive Control (MPC) has received considerable attention over the last decades and is currently a well-established control technique (see, e.g.,~\cite{bemporad2006model,kouvaritakis2016model,rawlings2020model, schwenzer2021review} for a comprehensive overview of its theoretical and practical aspects). 
Despite the growth in available computational power, the complexity of solving optimal control problems in real-time remains a critical challenge, particularly for high-dimensional or highly nonlinear models.
Thus, mitigating this limitation has become a central focus of current MPC research. 
In this paper, we address this challenge by framing two widely adopted strategies in practical applications within a rigorous methodological perspective: (i) model order reduction and (ii) suboptimal approaches.
Accordingly, we organize the literature review into two corresponding main parts.
\\
The core principle of model-order reduction techniques is to employ simplified yet effective models capturing key features of more complex ones (such as those arising in hierarchical and/or interconnected architectures). 
In this spirit, the work~\cite{rosolia2020multi} proposes a multi-rate architecture with a low-level component ensuring safety by acting on the full-order model, while a high-level MPC scheme optimizes performance on a reduced-order one.
The authors of~\cite{csomay2022multi} design a multi-rate scheme with a high-level component planning a continuous reference trajectory using a linearized MPC approach, a low-level component based on control Lyapunov functions applied to the full-order model, and Bézier curves to connect the two layers.
The paper~\cite{rosolia2022unified} generalizes the hierarchical approach by proposing an architecture combining a high-level Markov Decision Process that iteratively updates the constraint and cost functions of a mid-level MPC scheme providing setpoints to a low-level controller.
The work~\cite{picasso2010mpc} presents a two-layer MPC scheme designed by following a robust control approach.
Similarly, \cite{zhang2022dual} proposes a dual-level MPC scheme based on two regulators operating at different frequencies.
In the process control domain, \cite{chen2011model} develops a composite control framework based on multirate sampling, combining a fast feedback controller for the stabilization of the fast dynamics with a slow MPC scheme managing the slow subsystem.
This paradigm, widely adopted in robotic applications, has been formalized in the tutorial work~\cite{matni2024towards}.
In~\cite{loehning2014model}, the authors establish closed-loop system stability despite the use of a reduced model in the MPC optimizer.
A robust, reduced-order MPC scheme is presented in~\cite{lorenzetti2019reduced}, where stability and constraint satisfaction are proven while accounting for model reduction errors, whose bounds can be computed, e.g., as in~\cite{lorenzetti2020error}.
In~\cite{kartmann2024certified}, Galerkin projection methods are employed to operate on a low-dimensional surrogate model.
The authors of~\cite{wang2022tube} present a robust MPC framework for two-time-scale linear systems, where only a reduced-order model neglecting the fast dynamics is used in the MPC optimizer.
Recently, a novel approach has been proposed to extract low-dimensional control-oriented models from data using Spectral Submanifold Reduction techniques~\cite{alora2023data}.
Other approaches focus on the discretization of the prediction horizon to reduce complexity. 
For instance, \cite{tan2016model} reduces decision variables in coupled multi-scale dynamics by employing a non-uniformly spaced horizon,
while~\cite{schroth2025multi} leverages exponential decay of sensitivity to progressively reduce model complexity along the horizon
by capitalizing on the diminishing impact of distant modeling inaccuracies on the current control action.
\\
Complementary to model reduction, suboptimal MPC has emerged as a prominent strategy to alleviate the computational burden.
This paradigm consists of computing only an approximate solution to the optimal control problem at each sampling time, rather than solving it to optimality (see, e.g.,~\cite{diehl2005real,gros2020linear}).
Foundational studies establishing the theoretical viability of this approach include \cite{scokaert2002suboptimal,graichen2010stability,rubagotti2014stabilizing}.
In particular, the work~\cite{scokaert2002suboptimal} shows that stability can be guaranteed with suboptimal schemes, provided that a feasible initial solution is available.
In the field of continuous-time MPC, authors in~\cite{graichen2010stability} establish a lower bound on the number of optimization steps performed at each sampling instant to achieve incremental improvement. %
In~\cite{rubagotti2014stabilizing}, stability guarantees are provided for MPC schemes applied to linear systems, in which, at each iteration, the associated quadratic program is neither solved to optimality nor satisfies the inequality constraints.
In~\cite{zanelli2021lyapunov}, a suboptimal MPC scheme is proposed with theoretical guarantees on closed-loop stability obtained by deriving an upper bound on the sampling time.
The authors of~\cite{karapetyan2023finite} study the finite-time behavior of suboptimal linear MPC and provide insights into the trade-offs between computational efficiency and control performance.
This work is extended in~\cite{karapetyan2025closed} to nonlinear systems.
More recently,~\cite{chen2025sampled} establishes the existence of a sampling-time bound under which the closed-loop system achieves exponential stability.
In~\cite{preuster2026optimization,preuster2026coupling}, an
alternative approach is explored by using a port-Hamiltonian framework
to analyze the stability of the closed-loop system resulting from
suboptimal MPC schemes.

The main contribution of this paper is the development of
SuboptiMAl and Reduced-order Tunable MPC (SMART-MPC), a
generalized framework to design and analyze MPC architectures that simultaneously
leverage suboptimal optimization and model-order
reduction. %
Specifically, we consider continuous-time nonlinear systems described by the feedback interconnection of a primary target dynamics (whose state enters the cost function of the underlying optimal control problem), and an auxiliary extra dynamics (representing unmodeled dynamics and/or a low-level compensator), whose time scale can be tuned by adjusting a suitable parameter.
Within this framework, we appropriately design the sampling time (see, e.g.,~\cite{zanelli2021lyapunov,chen2025sampled}) of the underlying continuous-time plant so as to enable the use of an MPC scheme that (i) performs only a finite number of iterations of a generic optimization algorithm between two sampling instants,  and (ii) relies only on a reduced-order model of the plant, obtained by constraining the extra dynamics to an associated equilibrium manifold.
By using a timescale-separation approach, we show that the closed-loop system resulting from the interconnection of the suboptimally solved reduced-order MPC problem with the full-order plant enjoys recursive feasibility and has an exponentially stable equilibrium point.
At this equilibrium, (i) the target dynamics reaches its desired setpoint, (ii) the optimizer state is such that the corresponding input estimate coincides with the optimal input sequence, and (iii) the extra dynamics lies on its equilibrium manifold.
Thus, we leverage timescale separation not only to address the distinct time scales naturally arising in many physical systems (see, e.g., the survey~\cite{abdelgalil2023multi}), but also to capture the interaction between the algorithmic component and the plant, as recently explored in the context of feedback optimization (see, e.g.,~\cite{colombino2019online,hauswirth2020timescale,belgioioso2024online,carnevale2024nonconvex}).
\\
A preliminary short version of this work appeared in~\cite{di2025nonlinear}. 
Therein, the plant is modeled as a discrete-time system and the sampling time is not explicitly analyzed.
By contrast, the present paper provides rigorous proofs starting directly from the continuous-time dynamics and explicitly accounts for the role of the sampling time. 
Furthermore, this work introduces a formal guarantee of recursive feasibility and validates the proposed approach through high-fidelity virtual experiments on a more challenging robotic setup. 

The paper is organized as follows. 
Section~\ref{sec:problem_formulation} describes the problem setting.
Section~\ref{sec:algo} details the proposed \algo/ algorithm and states the stability properties of the closed-loop system.
The theoretical analysis is provided in Section~\ref{sec:theoretical_analysis}.
Finally, Section~\ref{sec:simulations} validates the theoretical results through virtual experiments.

\vspace{-.06cm}

\paragraph*{Notation} 

$\R$ and $\N$ denote the sets of real and natural numbers, while $\R\ud{+}$ is the set of non-negative real numbers. 
We denote by $\col(v_1,\dots,v_N)$ the vertical concatenation of the vectors $v_1,\dots,v_N$ and by $\diag(d_1, \dots, d_N)$ the block-diagonal matrix with blocks $d_1, \dots, d_N$ on the diagonal. 
Given $V: X \to \R$, we denote by $\Omega_V(\gamma) = \{x \in X\mid V(x) \leq \gamma\}$ the sublevel set of $V$ at level $\gamma \in \R$.
\section{Scenario Description}
\label{sec:problem_formulation}

In this section, we describe the class of systems we consider in this work.
Specifically, we focus on interconnected nonlinear dynamical systems modeled in continuous time as
\begin{subequations}\label{eq:plant_ct}
    \begin{align}
        \xcontdot(\tti) &= \slowc(\xcont(\tti), \xicont(\tti), \ucont(\tti))\label{eq:slow_plant_ct}\\
        \prct \xicontdot(\tti) &= \fastc(\xcont(\tti), \xicont(\tti), \ucont(\tti)),\label{eq:fast_plant_ct}
    \end{align}
\end{subequations}
where $\xcont(\tti) \in \domX \subseteq \R^{n}$ is the target state at time $\tti \in \R$, $\xicont(\tti) \in \domXi \subseteq \R^{p}$ is the extra dynamics state and $\ucont(\tti) \in \domU \subseteq \R^{m}$ is the control input at time $\tti \in \R$.
The functions $\slowc: \domX \times \domXi \times \domU \to \R^n$ and $\fastc: \domX \times \domXi \times \domU \to \R^p$ describe the target and extra vector fields, respectively, and are continuously differentiable.
The set $\domX \times \domXi$ is forward invariant for system~\eqref{eq:plant_ct}.
Finally, $\prct > 0$ is a parameter that allows for arbitrarily accelerating the extra dynamics~\eqref{eq:fast_plant_ct}.
To bridge the gap between the continuous-time physical model and the discrete-time MPC framework, we consider a sampled-data representation of the system. 
To this end, we use $\iter \in \N$ to denote the discrete-time index, and we let $\xt = \xcont(\tti_{\iter})$, $\xit = \xicont(\tti_{\iter})$, and $\ut = \ucont(\tti_{\iter})$ be the sampled states and input at time $\tti = \tti_\iter := \iter\pr$, where $\pr > 0$ denotes the sampling time and we treat it as a tunable parameter. 
In line with typical MPC schemes, we consider a zero-order hold on the input such that $u(\tti) = \ut$ for all $\tti \in [\tti_{\iter}, \tti_{\iter} + \pr)$. %
Then, we express the discrete-time evolution of system~\eqref{eq:plant} as
\begin{subequations}\label{eq:plant}
    \begin{align}
        \xtp &= \slow(\xt, \xit, \ut, \pr)
        \label{eq:slow_plant}
        \\
        \xitp &= \fast(\xt, \xit, \ut, \pr),
        \label{eq:fast_plant}
    \end{align}
\end{subequations}
where $\slow: \domX \times \domXi \times \domU \times \R\ud{+} \to \domX$ and $\fast: \domX \times \domXi \times \domU \times \R\ud{+} \to \domXi$ are the discrete-time vector fields obtained from the continuous-time dynamics by sampling at time intervals of length $\pr$, namely 
\begin{subequations}\label{eq:slow_fast_definition}
    \begin{align}
        \slow(x, \xi, \uu, \pr) 
        &
        := x 
        + \int_{0}^{\pr} \slowc(\flow(x,\xi,u,\tint), \uu) \, d\tint
        \label{eq:slow_definition}
        \\
        \fast(x, \xi, \uu, \pr) 
        &:= \xi + \int_{0}^{\pr} \frac{1}{\prct}\fastc(\flow(x, \xi, u, \tint),\uu)\, d\tint,
        \label{eq:fast_definition}
    \end{align}
\end{subequations}
where $\flow(x, \xi, u,\tint) \in \domX \times \domXi$ denotes the solution to the continuous-time system~\eqref{eq:plant_ct} initialized at $(x,\xi) \in \domX \times \domXi$ and subject to the constant input $u(\tint) = u \in \domU$ for all $\tint \in [0, \pr)$.

\begin{assumption}\label{ass:slowness}
    There exist $L_{\slowc}, L_{\fastc} > 0$ such that %
    \begin{align*}
        \norm{\slowc(x, \xi, u) - \slowc(x^\prime,\xi^\prime,u^\prime)}  &\leq L_{\slowc}\norm{\col(
            x -  x^\prime,
            \xi  -  \xi^\prime,
            u  -  u^\prime)}
        \\
        \norm{\fastc(x, \xi, u) - \fastc(x^\prime,\xi^\prime,u^\prime)}  &\leq  L_{\fastc}\norm{\col(
            x  -  x^\prime,
            \xi  -  \xi^\prime,
            u  -  u^\prime)},
    \end{align*}
    for all $x, x^\prime \in \domX$, $\xi, \xi^\prime \in \domXi$, and $u, u^\prime \in \domU$.
    \oprocend
\end{assumption}
As we formalize in the following assumption, we consider the case in which the extra dynamics~\eqref{eq:fast_plant_ct} has equilibria parametrized in $(x,u)$ and that these equilibria are globally exponentially stable, uniformly in $(x,u)$, in the limit case in which $x$ and $u$ are fixed.
\begin{assumption}\label{ass:ges_fast}
    There exists a $\lipeq$-Lipschitz continuous function $\xieq: \domX \times \domU \to \domXi$ such that
    \begin{align}\label{eq:xieq}
        \fastc(x, \xieq(x, u), u) = 0,
    \end{align}
    for all $x \in \domX$, $u \in \domU$.
    Moreover, there exists a continuously differentiable function $\cG: \R^{p} \to \R_+$ with $a_4$-Lipschitz continuous gradient $\nabla \cG$ such that it holds
    \begin{subequations}\label{eq:U}
        \begin{align}
            &a_1  \snorm{\xi  -  \xieq(x, u)}^2  \leq  \cG(\xi  -  \xieq(x, u))  \leq  a_2 \snorm{\xi  -  \xieq(x, u)}^2
            \label{eq:U_1}
            \\[0.5em]
            &\nabla \cG(\xi  -  \xieq(x, u))^\top \fastc(x,\xi, u) \leq -a_3 \snorm{\xi  -  \xieq(x, u)}^2,
            \label{eq:U_2}
        \end{align}
    \end{subequations}
    for all $\xi \in  \domXi$, $x\in \domX$, $u \in \domU$ and some constants $a_1, a_2, a_3, a_4 > 0$.
    \oprocend
\end{assumption}
The practical intuition behind this assumption is that the extra state $\xi$ exponentially converges to a corresponding steady-state configuration $\xieq(x, u)$ in the ideal case in which $x$ and $u$ are fixed.
Indeed, we remark that the left-hand side of~\eqref{eq:U_2} corresponds to the Lie derivative of $\cG$ along the vector field of the continuous-time extra dynamics~\eqref{eq:fast_plant_ct} using the dilated-time variable $\dt := \tti/\prct$ and an arbitrarily fixed pair $(x,u)$.
We provide as follows a practical example satisfying Assumption~\ref{ass:ges_fast}.

\subsection{Underactuated Robotic Arm with DC Motor} 
\label{sec:example}

We consider the PenduBot~\cite{zhang2002hybrid}: an underactuated two-link robotic arm driven by a DC motor at the first joint. 
Let $q = \left[\vartheta_1, \vartheta_2\right] \in \R^2$ be the vector of generalized coordinates and $\tau_m \in \R$ the first joint torque. 
Then, 
the equation of motion reads as
\begin{align}\label{eq:mechanical_dynamics}
    M(q) \ddot{q} + C(q, \dot{q}) \dot{q} + B \dot{q} + G(q)  =  \begin{bmatrix}
            \tau_m&
            0
    \end{bmatrix}\T,
\end{align}
with 
\begin{align*}
    M(q) &:= \begin{bmatrix}
        p_1 + 2 p_2 c_2 & p_3 + p_2 c_2 \\
        p_3 + p_2 c_2 & p_3
    \end{bmatrix}, & G(q) &:= \begin{bmatrix}
        p_4 s_1 + p_5 s_{12} \\
        p_5 s_{12}
    \end{bmatrix}, \quad
    C(q, \dot{q})\dot{q} := \begin{bmatrix}
        -p_2 s_2 \dot{q}_2 (2 \dot{q}_1 + \dot{q}_2)\\
        p_2 s_2 \dot{q}_1^2
    \end{bmatrix}, & B := \begin{bmatrix}
        b_1 & 0 \\
        0 & b_2
    \end{bmatrix},
\end{align*}
in which $c_2 := \cos(q_2)$, $s_2 := \sin(q_2)$, $s_1 := \sin(q_1)$, $s_{12} := \sin(q_1 + q_2)$, and
\begin{align*}
    p_1 = J_1 + J_2 + m_1 l_{c,1}^2 + m_2 (l_1^2 + l_{c,2}^2), \quad p_2 = m_2 l_1 l_{c,2}, \quad
    p_3 = J_2 + m_2 l_{c,2}^2, \quad p_4 = m_1 g l_{c,1} + m_2 g l_1,  \quad p_5 = m_2 g l_{c,2},
\end{align*}
where $m_i, l_i, l_{c,i}, J_i$ are the mass, length, center of mass distance, and moment of inertia of the $i$-th link, respectively, while $g$ is the gravity acceleration.
The electrical dynamics of the armature circuit reads as
\begin{align}\label{eq:dc_motor}
    L \dot{I} + R I + K_e \omega_1 = V,
\end{align}
where $L$ and $R$ are the armature inductance and resistance, respectively, $K_e > 0$ is the back-EMF constant, and $V \in \R$ is the applied voltage. 
The torque vector generated by the actuation system is given by $\tau_m = K_t I$, where $K_t > 0$ is the motor torque constant. 
To write the state space model, we define $\xcont = \left[\theta\ud{1, c}, \theta_{2,c}, \omega_{1,c}, \omega_{2,c}\right] \in \R^4$, consider the armature current $I \in \R$ as an extra state $\xicont$, and the applied voltage $V \in \R$ as the input $\ucont$. 
The continuous-time state space model is then given by
\begin{subequations}\label{eq:pendubot_dynamics}
    \begin{align}
        \xcontdot &= \begin{bmatrix}
            \x_{3,c} \\
            \x_{4,c} \\
             M^{-1}(\xcont) \left(-\left(C(\xcont)  +  B\right)\begin{bmatrix} \x_{3,c} \\ \x_{4,c} \\ \end{bmatrix} - G(\xcont) + \begin{bmatrix}
                 K_t \xicont
                \\
                0
            \end{bmatrix}\right)
        \end{bmatrix}\label{eq:pendubot_dynamics_slow}
        \\
        L\xicontdot &= -R \xicont - K_e \x_{3,c} + \ucont.\label{eq:pendubot_dynamics_fast}
    \end{align}
\end{subequations}
Thus, system~\eqref{eq:pendubot_dynamics} is an instance of the generic system~\eqref{eq:plant_ct} in which the role of $\prct$ is played by the armature inductance $L$, while $\slowc$ and $\fastc$ are defined as the right-hand sides of \eqref{eq:pendubot_dynamics_slow} and \eqref{eq:pendubot_dynamics_fast}, respectively.
We show as follows that 
Assumption%
~\ref{ass:ges_fast} is satisfied. 
Subsystem~\eqref{eq:pendubot_dynamics_fast} admits the equilibrium function $\xieq: \domX \times \domU \to \domXi$ given by
\begin{align}
    \xieq(x, u) = 
    \tfrac{1}{R} \left(
        - \begin{bmatrix}0 & 0 &K_e & 0\end{bmatrix} x + u
    \right).\label{eq:manifold_ex}
\end{align}
To show exponential stability of~\eqref{eq:manifold_ex}, we pick the candidate Lyapunov function $\cG(\xi - \xieq(x, u)) = \frac{1}{2} (\xi - \xieq(x,u))^2$, which trivially satisfies~\eqref{eq:U_1}. %
For an arbitrarily fixed pair $(x,u)$ and using the dilated-time variable $\dt := \tti/\prct$, the Lie derivative of $\cG$ along system~\eqref{eq:pendubot_dynamics_fast} is
\begin{align*}
    &\nabla G(\xi - \xieq(x, u))^\top  (-R \xi - \begin{bmatrix}0 & 0 &K_e & 0\end{bmatrix} x+ u)
    = - R (\xi - \xieq(x, u))^2.
\end{align*}
Since $R > 0$, this guarantees that also~\eqref{eq:U_2} is satisfied.
\begin{remark} 
    As suggested by the generality of the equations in~\eqref{eq:mechanical_dynamics}, we note that this approach can be extended to any mechanical system admitting such a formulation and controlled by an auxiliary system (e.g., the DC motor considered above), whose speed can be arbitrarily tuned.
    \oprocend
\end{remark}

\section{Suboptimal and Reduced-Order Tunable MPC: SMART-MPC}
\label{sec:algo}

In this section, we introduce the proposed \algoCAPS/ (\algo/) architecture and state the corresponding stability guarantees for the resulting closed-loop system.

\subsection{\algo/: Description}

By exploiting the system structure outlined in Section~\ref{sec:problem_formulation}, we design an MPC strategy that operates on a reduced-order model of system~\eqref{eq:plant}.
Specifically, we approximate the full dynamics by a reduced dynamics map $\redc: \domX \times \domU \to \R^n$
that captures the target vector field $\slowc$ evaluated on the extra dynamics manifold $\xicont = \xieq(\xcont, \ucont)$, namely
\begin{align}\label{eq:redc}
    \redc(\xcont, \ucont) := \slowc(\xcont, \xieq(\xcont, \ucont), \ucont).
\end{align}
Consistently with the computational-efficiency rationale motivating the use of $\redc$, the input $\ut$ is iteratively computed through a \emph{suboptimal} MPC scheme applied to a discretized version of the reduced-order system, namely 
\begin{align*}
    \xtp = \redd(\xt, \ut),
\end{align*}
where $\redd$ is a discrete-time approximation of $\dot{x}\ct = \redc(\xcont, \ucont)$ (obtained, e.g., through a Runge-Kutta scheme).
\begin{figure}[htpb]
    \centering
    \includegraphics[width=0.9\columnwidth]{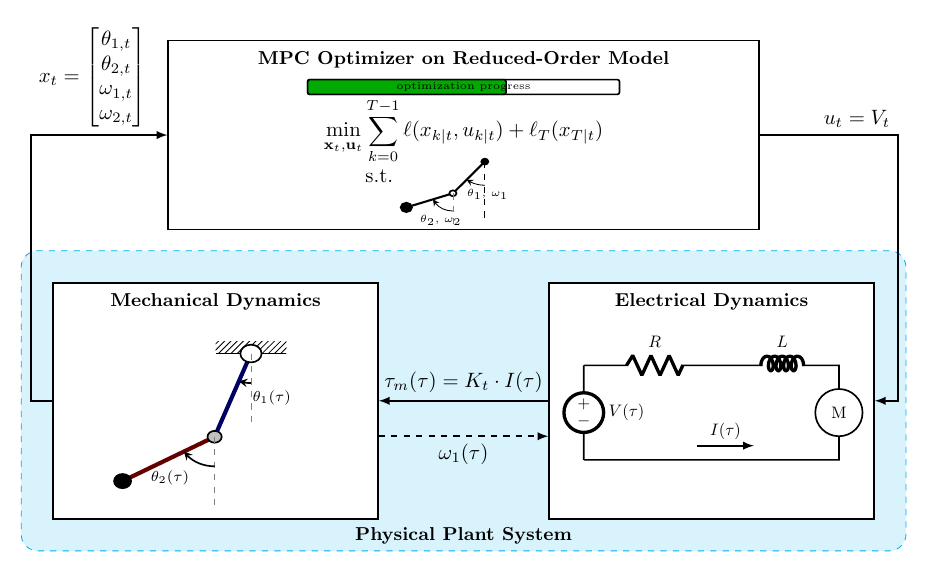}
    \caption{Block diagram representation of the closed-loop system resulting from the interconnection of the full plant dynamics~\eqref{eq:plant} and the \algo/~\eqref{eq:mpc_dyn} for the PenduBot-motor setup (cf.~Section~\ref{sec:example}).}
    \label{fig:bd_pendulum}
\end{figure}
In detail, given a time horizon $\hor \in \N$, at each $\iter \in \N$, we consider the finite-horizon optimal control problem
\begin{align}\label{eq:MPC_problem}
    \begin{aligned}
        \min_{\mathbf{x}_\iter, \mathbf{u}_\iter} \quad\quad & %
        \sum_{\itermpc=0}^{\hor-1}\cost\left(\x_{\itermpc \mid \iter}, u_{\itermpc \mid \iter}\right) + \costf \left(\x\ud{\hor \mid \iter}\right) 
        \\
        \text{s.t.} \quad\quad
            & x\ud{0 \mid \iter} = \xt
            \\%
            & \x_{\itermpc+1 \mid \iter} = \redd(\x_{\itermpc \mid \iter}, u_{\itermpc \mid \iter}),  &&\itermpc = 0, \ldots, \hor-1
            \\
            & u_{\itermpc \mid \iter} \in \mathcal{U}, \quad  \x_{\itermpc \mid \iter} \in \mathcal{X}, 
            &&\itermpc = 0, \ldots, \hor-1 
            \\
            & x\ud{\hor \mid \iter} \in \mathcal{X}_f, %
    \end{aligned}
\end{align}
where $\hor \in \N$ is the prediction horizon, $\cost: \domX \times \domU \to \R\ud{+}$ is the stage cost, $\costf: \domX \to \R\ud{+}$ is the terminal cost, $\mathcal{U} \subseteq \domU$ is the input constraint set, $\mathcal{X} \subseteq \domX$ is the state constraint set, $\hor$ is the prediction horizon, $\mathcal{X}_f \subseteq \domX$ is the terminal constraint set, $\mathbf{x}_\iter = \left\{x_{0\mid \iter}, \dots, \x\ud{\hor \mid \iter}\right\}$ and $\mathbf{u}_\iter = \left\{u_{0\mid \iter}, \dots, u_{\hor-1 \mid \iter}\right\}$ are the state and input optimization sequences, being $x_{\itermpc \mid \iter}$ and $u_{\itermpc \mid \iter}$ the predicted state and input at the optimization step $\itermpc$ within the horizon at time step $\iter$.
Problem~\eqref{eq:MPC_problem} is addressed suboptimally.
For the sake of generality, we do not specify the particular optimization scheme used to solve problem~\eqref{eq:MPC_problem} suboptimally, but we rather assume that such a scheme can be described by a discrete-time dynamical system of the form
\begin{subequations}\label{eq:mpc_dyn}
    \begin{align}
        \ztp &= \mpc(\zt, \xt) 
        \\
        \ut &= \Pi(\zt),
    \end{align}
\end{subequations}
where $\zt \in \domZ \subseteq \R^{\dimz}$ is the optimizer state, $\mpc: \domZ \times \domX \to \domZ$ is the optimal control algorithm and $\Pi: \domZ \to \domU$ is a projection map that extracts only the first component of the current optimal input sequence estimate $\mathbf{u}_\iter$ from the overall optimizer state $\zt$. 
Indeed, the optimizer state $\zt$ does not necessarily coincide with the current estimate of the optimal input sequence $\mathbf{u}_\iter$. 
Depending on the specific optimization scheme considered, $\zt$ may also include the predicted state sequence $\mathbf{x}_\iter$ and/or auxiliary variables used by the optimization algorithm, such as Lagrange multipliers or momentum terms.
Algorithm \ref{alg:mpc_subopt} summarizes the steps of the proposed \algo/ strategy.
Fig.~\ref{fig:bd_pendulum} shows a schematic representation of the closed-loop system in the PenduBot-motor setup described in Section~\ref{sec:example}.
\begin{algorithm}[t]
    \caption{\algo/}
    \label{alg:mpc_subopt}
    \begin{algorithmic}[1]

        \For{$t = 0, 1, 2, \ldots$}
       \State \hspace{-1.64em} Full-order plant actuation:
            \vspace{-1em}
            \begin{align*}
                \ut &= \Pi(\zt)
                \\
                \xtp &= \slow(\xt, \xit, \ut, \pr)
                \\
                \xitp &= \fast(\xt, \xit, \ut, \pr)
            \end{align*}
        \vspace{-1.5em}
         \State \hspace{-1.64em} MPC optimizer step on reduced-order problem~\eqref{eq:MPC_problem}:
            \vspace{-1em}
            \begin{align*}
                \ztp &= \mpc(\zt, \xt)
            \end{align*}
        \vspace{-1.5em}
    \EndFor
\end{algorithmic}
\end{algorithm}

\subsection{\algo/: Recursive Feasibility and Stability}
\label{sec:algorithm_presentation}

This subsection details the main result of this paper on the closed-loop system arising from the interconnection of the plant~\eqref{eq:plant} and the optimizer~\eqref{eq:mpc_dyn}, namely
\begin{subequations}\label{eq:interconnected_sys}
    \begin{align}
        \xtp &= \slow(\xt, \xit, \Pi(\zt), \pr)
        \label{eq:interconnected_system_x}
        \\
        \xitp &= \fast(\xt, \xit, \Pi(\zt),\pr)
        \label{eq:interconnected_system_xi}
        \\
        \ztp &= \mpc(\zt, \xt).
        \label{eq:interconnected_system_z}
    \end{align}
\end{subequations}
In order to provide this result, some assumptions are detailed in the following.
To this end, we denote as $\feas \subseteq \domX$ the set of initial target states for which problem~\eqref{eq:MPC_problem} is feasible on the horizon $\hor$, namely, $\feas$ reads as
\begin{align}
    \feas := \{x \in \cX \mid &\exists {u\ud{0},\dots,u\ud{\hor-1}} \in \cU \text{ s.t. } x\ud{\hor} \in \cX_f,
    %
    \x\ud{\itermpc+1} = \redd(\x\ud{\itermpc},\uu\ud{\itermpc}), \itermpc = 0, \dots, \hor-1,
    \notag\\
    &\x\ud{\itermpc} \in \cX, \itermpc = 0, \dots, \hor-1,
    \x\ud0 = x\}.\label{eq:cX_N}
\end{align}
Then, 
we assume as follows that the optimizer dynamics~\eqref{eq:mpc_dyn} has a globally exponentially stable equilibrium for each initial condition $x \in \feas$. 
In this regard, it is important to note that, in the actual closed-loop system~\eqref{eq:interconnected_sys}, at each iteration $\iter$, the optimizer dynamics~\eqref{eq:mpc_dyn} focuses on an updated version of problem~\eqref{eq:MPC_problem} with initial condition corresponding to the current state $\xt$ of the plant.
\begin{assumption}\label{ass:ges_MPC}
    There exists $\zstar: \feas \to \domZ$ such that
    \begin{align}\label{eq:optimum_equilibrium}
        \zstar(x) = \mpc(\zstar(x), x).
    \end{align}
    Moreover, there exists a continuously differentiable function $\cL: \R^{\dimz} \to \R\ud{+}$ with $d_4$-Lipschitz continuous gradient $\nabla \cL$ such that
    \begin{subequations}\label{eq:cL}
            \begin{align}
            &d_1\norm{z - \zstar(x)}^2 \leq \cL(z - \zstar(x)) \leq d_2\norm{z - \zstar(x)}^2
            \label{eq:cL_1}
            \\
            &\cL(\mpc(z,x) - \zstar(x)) - \cL(z - \zstar(x)) \leq - d_3\norm{z - \zstar(x)}^2,
            \label{eq:cL_2}
        \end{align}
    \end{subequations}
    for all $z \in \domZ$, $x \in \feas$, and some $d_1,d_2,d_3,d_4 > 0$.
    Further, $\zstar$, $\mpc$, and $\Pi$ are $\lip_{\zstar}$-, $\lip_{\mpc}$- ,and $\lip_{\Pi}$-Lipschitz continuous, respectively, for some $\lip_{\zstar}, \lip_{\mpc}, \lip_{\Pi} > 0$.
    \oprocend
\end{assumption}
With the properties of the optimization algorithm established, we now turn our attention to the corresponding control problem.
In particular, we guarantee that the reduced model (see~\eqref{eq:redc}) controlled through the optimal policy $u = \Pi(\zstar(x))$ arising from problem~\eqref{eq:MPC_problem} has an exponentially stable equilibrium $\xstar \in \feas$.  %
\begin{assumption}\label{ass:red_ges}
    There exists $\xstar \in \feas$ such that 
    \begin{align*} 
        \redc(\xstar, \Pi(\zstar(\xstar))) =  0.
    \end{align*}  
    Moreover, there exists a continuously differentiable function $W: \feas \to \R_+$ with $c_4$-Lipschitz continuous gradient $\nabla W$such that
    \begin{subequations}\label{eq:V_properties}
        \begin{align}
            c_1 \norm{x-\xstar}^2 \leq W(x) &\leq c_2 \norm{x-\xstar}^2
            \label{eq:quadratic_bounds_mpc}
            \\
            \nabla W(x)^\top \redc(x, \Pi(\zstar(x))) &\leq - c_3 \norm{x-\xstar}^2,
            \label{eq:descent_property_mpc}
        \end{align} 
    \end{subequations}
    for all $x\in \feas$ and some $c_1,c_2,c_3,c_4 > 0$.
    \oprocend
\end{assumption}
\begin{remark}
    Sufficient conditions on the underlying optimal control problem~\eqref{eq:MPC_problem} can be established to guarantee Assumption~\ref{ass:red_ges}, namely, exponential stability of $\xstar$ for the reduced-order continuous-time dynamics $\xcontdot = \redc(\xcont,\ucont)$ under the optimal MPC law $\ucont = \Pi(\zstar(\xcont))$.
    These conditions regard, e.g., appropriate choice of stage and terminal cost functions, terminal constraints, sufficiently long prediction horizon, and regularity conditions on the system dynamics, see, e.g.,~\cite[Chap.~2, pp.~139--141]{rawlings2020model}.
    \oprocend
\end{remark}

Then, by recalling the definition of the feasible set $\feas$ (cf.~\eqref{eq:cX_N}), we introduce the parameter $\gammam > 0$ to characterize the largest level set of the Lyapunov function $W$ (cf. Assumption~\ref{ass:red_ges}) contained in $\feas$.
Namely, we define %
\begin{align}
    \gammam := \sup\{\gamma \in \R_+ \mid \Omega_W(\gamma) \subseteq \feas\}.\label{eq:gammam}
\end{align}
Now, we can formalize the recursive feasibility and stability properties of the closed-loop system~\eqref{eq:interconnected_sys}.
To this end, we introduce the Lyapunov function $V: \feas \times \domXi \times \domZ \to \R\ud{+}$ defined as
\begin{align}
    V(x, \xi, z) &:= W(x) + \cG(\xi - \xieq(x, \Pi(z)))
    + \kappa \cL(z - \zstar(x)) ,
    \label{eq:V_pre_thm}
\end{align}
where $\cG$, $\cL$, and $W$ are the functions introduced in Assumptions~\ref{ass:ges_fast},~\ref{ass:ges_MPC}, and~\ref{ass:red_ges}, respectively, while $\kappa > 0$ is a positive constant that will be characterized in the analysis.
\begin{theorem}\label{th:main}
    Consider the closed-loop system~\eqref{eq:interconnected_sys} and let Assumptions~\ref{ass:slowness},~\ref{ass:ges_fast},~\ref{ass:ges_MPC}, and~\ref{ass:red_ges} hold. 
    Then, for all $\gamma \in (0,\gammam)$ and $\ratio > 0$, there exist $\bar{\pr}, \bar{\kappa} > 0$ such that, for all $\pr \in (0, \bar{\pr})$, $\prct = \ratio\pr$, and $\kappa \geq \bar{\kappa}$, the following properties hold:
    \begin{itemize}
        \item \textbf{(Recursive Feasibility):}
        the set $\Omega_V(\gamma)$ is forward invariant for system \eqref{eq:interconnected_sys}, implying that $\xt \in \feas$ for all $\iter\in\N$.
        \item \textbf{(Stability):} $(\xstar,\xieq(\xstar,\Pi(\zstar(\xstar))),\zstar(\xstar))$ is an exponentially stable equilibrium point of system~\eqref{eq:interconnected_sys} within the forward invariant set $\Omega_V(\gamma)$.\oprocend
    \end{itemize}
\end{theorem}
The proof of Theorem~\ref{th:main} is provided in Section~\ref{sec:proof}.
\begin{figure*}[htpb]
  \centering
  \begin{subfigure}[b]{0.48\textwidth}
    \centering
    \includegraphics[width=\textwidth]{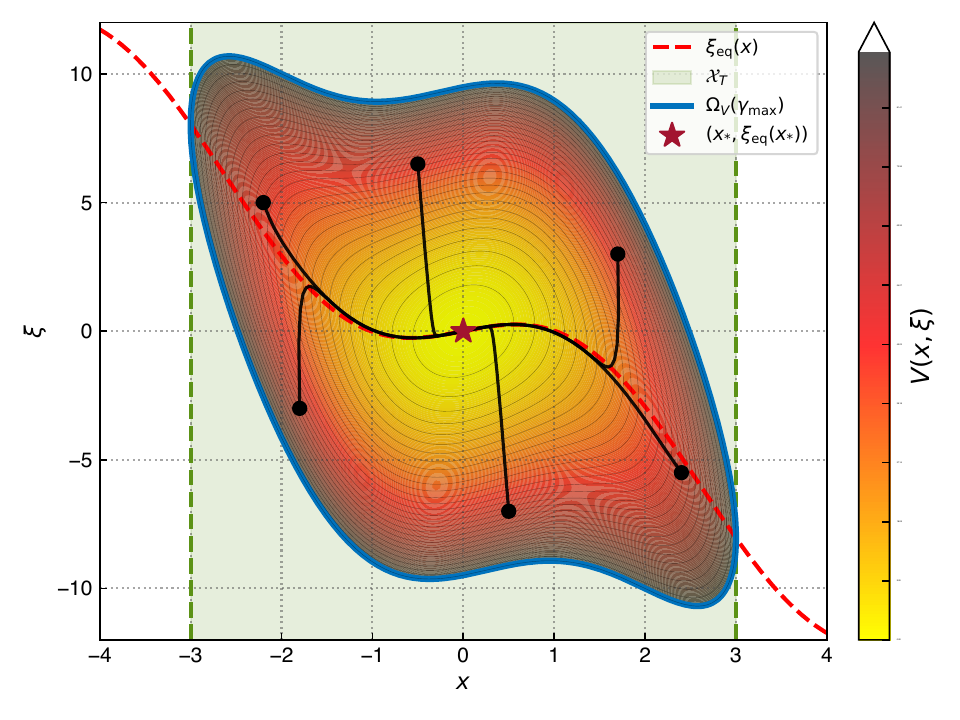}
    \caption{}
    \label{fig:2d_levelsets}
  \end{subfigure}
  \hfill
  \begin{subfigure}[b]{0.48\textwidth}
    \centering
    \includegraphics[width=\textwidth]{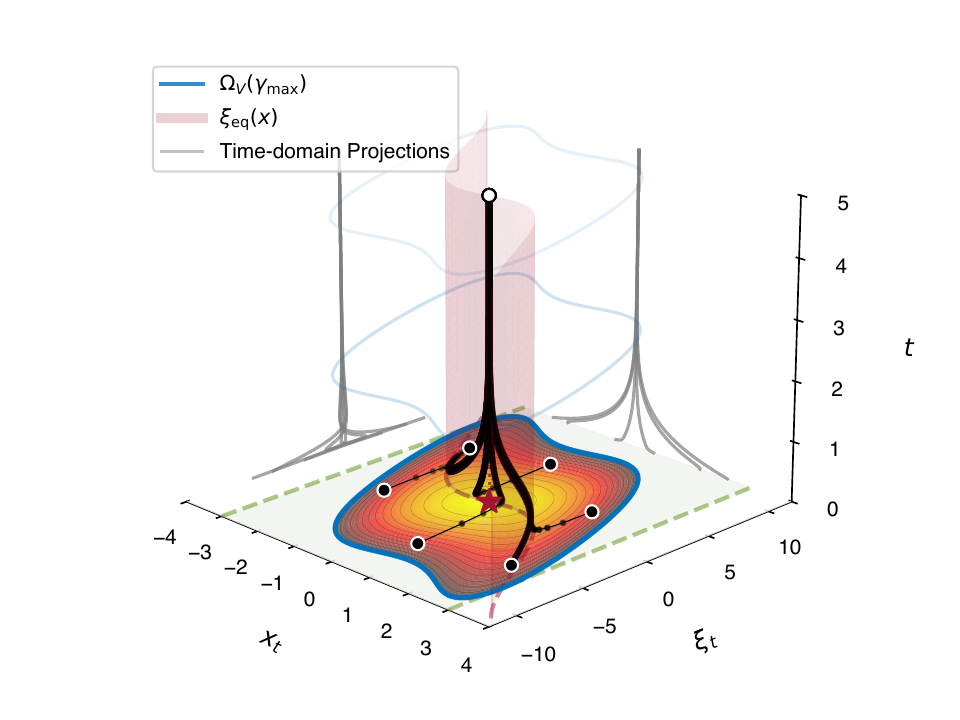}
    \caption{}
    \label{fig:3d_spacetime}
  \end{subfigure}
\caption{(a) Phase-portrait of the toy example showing the maximum forward invariant set $\Omega_V(\gammam)$ (thick blue line), tangent to the boundary of the feasible set $\feas$ (green area within $x \in [-\pi, \pi]$). (b) Space-time evolution highlighting the timescale separation, with fast boundary-layer convergence followed by slower target motion along the equilibrium manifold. (red surface).} 
  \label{fig:cbf_comparison}
\end{figure*}

\subsection{Discussion of Theorem~\ref{th:main} via a Toy Example}

Theorem~\ref{th:main} guarantees recursive feasibility and stability for the closed-loop system~\eqref{eq:interconnected_sys} in $\Omega_V(\gamma)$.
We recall that $\gammam$ (see~\eqref{eq:gammam}) characterizes the largest level set of the Lyapunov function $W$ (cf. Assumption~\ref{ass:red_ges}) that is contained in the feasible set $\feas$ (see~\eqref{eq:cX_N}).
Since $V$ (cf.~\eqref{eq:V_pre_thm}) is the candidate Lyapunov function for the closed-loop system~\eqref{eq:interconnected_sys}, the latter must be initialized in $\Omega_V(\gamma)$ to ensure that $\x_0 \in \feas$. 
In detail, these properties are ensured for a sufficiently small $\pr$, that is, when the plant is sampled sufficiently fast (see~\eqref{eq:slow_fast_definition}).
To provide a graphical intuition of these results, we consider the %
feedback interconnected system
\begin{subequations}\label{eq:example_ct}
    \begin{align}
        \dot{x}\ct(\tti) &= \xicont(\tti) - \alpha \sin(x\ct(\tti))
        \label{eq:example_target_ct}
        \\
        \prct\dot{\xi}\ct(\tti) &= \ucont(\tti) + \alpha \sin(x\ct(\tti)) - \xicont(\tti),
        \label{eq:example_extra_ct}
    \end{align}
\end{subequations}
with $\xcont(\tti), \xicont(\tti), \ucont(\tti) \in \R$ and $\alpha, \prct > 0$.
The steady-state map of the extra subsystem \eqref{eq:example_extra_ct} reads as 
\begin{align}
    \xieq(x, u) = u + \alpha \sin(x),
\end{align}
which yields a reduced-order model corresponding to a single integrator $\dot{x}\ct(\tti) = u\ct(\tti)$. 
The control objective is to steer the target state $\x$ towards the origin $\xstar=0$.
For graphical clarity, we assume the MPC state $\zt$ always lies on its optimal manifold, i.e., $\zt = \zstar(\xt)$, for all $\iter \in \N$. 
In the optimal control problem, we consider a quadratic cost penalized over $\hor = 30$ steps, the input feasible set $[-u_{\max}, u_{\max}]$, and a Forward Euler discretization of the continuous-time dynamics~\eqref{eq:example_ct} with sampling time $\prm > 0$, giving rise to the discrete-time dynamics $\xtp = \xt + \prm\ut$.
Consequently, the feasible set $\feas$ (cf.~\eqref{eq:cX_N}) is given by the maximum state deviation steerable to the origin within $\hor$ steps under maximum control authority, namely $\feas = \{x \in \R \mid |x| \leq \hor \prm u_{\max}\}$. 
Following the discussion in the previous section, the composite Lyapunov function $V$ for the interconnected system can be constructed, e.g., as 
\begin{align}
    V(\x, \xi) = 10\x^2 + \txi^2.
\end{align}
Hence, the maximum level set contained in the feasible region is $\gammam = 10\x_{\max}^2$, with $\x_{\max} \in \partial \feas$. 
The behavior of the closed-loop system is illustrated in Fig. \ref{fig:cbf_comparison}. 
Specifically, Fig. \ref{fig:2d_levelsets} shows the level sets of the composite Lyapunov function $V(\x, \xi)$. 
The green area, bounded by the vertical dashed lines, represents the feasible set $\feas$. 
The dashed red curve represents the equilibrium manifold of the extra dynamics. 
The thick blue curve depicts the maximum invariant set $\Omega_V(\gammam)$.
In agreement with Theorem \ref{th:main}, we note that $\Omega_V(\gammam)$ is tangent to the boundary of $\feas$ strictly at points where $\xi$ tracks its equilibrium manifold. 
The black line demonstrates that randomly initialized closed-loop trajectories are successfully steered towards the origin without leaving the invariant set. 
Furthermore, Fig.~\ref{fig:3d_spacetime} shows the space-time evolution of the closed-loop system to explicitly illustrate the timescale separation. 
The nonlinear equilibrium manifold is depicted as a red surface extending through time.
 The 3D black curves capture the singular perturbation paradigm: an almost horizontal initial displacement (fast boundary-layer dynamics converging to the manifold in negligible time), followed by a slower, upward climbing motion along the surface towards the origin. 
 The blue curves show the forward invariant set. Finally, the grey curves on the lateral planes help to visualize the isolated behavior of the target and extra states.

\section{Theoretical Analysis}
\label{sec:theoretical_analysis}

This section proves Theorem~\ref{th:main} by analyzing the closed-loop system (Fig. \ref{fig:bd_interconnection}) arising from the interconnection of the plant~\eqref{eq:plant} and the MPC optimizer~\eqref{eq:mpc_dyn}.
We interpret this closed-loop system as a two-time-scale system: the fast subsystem comprises the plant extra dynamics~\eqref{eq:fast_plant} and the optimizer mechanism~\eqref{eq:mpc_dyn}, while the slow part isolates the target dynamics~\eqref{eq:slow_plant}, see Fig.~\ref{fig:bd_interconnection}.
\begin{figure}[H]
  \centering
  \includegraphics[scale=\scale]{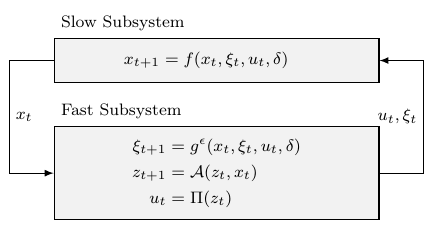}
  \caption{Block diagram representation of the interconnected system.}
  \label{fig:bd_interconnection}
\end{figure}
First, in Section~\ref{sec:single_integrator}, we show that system~\eqref{eq:interconnected_sys} enjoys some suitable properties related to the sampling time $\pr$ and Lipschitz continuity.
Then, as customary in two-time-scale system analysis, we study (i) the fast subsystem in the ideal case in which the slow state $\xt$ is fixed  (cf. Section~\ref{sec:bl}) and, conversely, (ii) the slow one in the ideal case in which the fast state $(\xit, \zt)$ lies in its equilibrium manifold (cf. Section~\ref{sec:reduced}).
Finally, in Section~\ref{sec:proof} we combine the results obtained in this preparatory phase to prove Theorem~\ref{th:main}.
Assumptions~\ref{ass:slowness}-\ref{ass:red_ges} hold true throughout the entire section.

\subsection{The Role of the Sampling Time $\pr$}
\label{sec:single_integrator}

First, we show that 
the variations of the slow dynamics $\slow$ can be controlled through the sampling time $\pr$, whereas those of the fast dynamics $\fast$ and $\mpc$ can be bounded uniformly with respect to $\pr$.
In this sense, we highlight that the sampling time $\pr$ plays the role of the tunable parameter required in two-time-scale analysis.
To compactly provide this result, we first introduce the discrete-time reduced dynamics $\redd: \domX \times \domU \times \R_+ \to \domX$ defined as
\begin{align}\label{eq:red}
    \red(x,u,\pr) := x + \int_{0}^{\pr}\redc(\flow_{\redc}(x,u,\tint),u)d\tint,
\end{align}
where $\flow_{\redc}(x,u,\tint) \in \R^{n}$ is the solution at time $\tint$ of the continuous-time reduced dynamics $\xcontdot = \redc(\xcont,\ucont)$ with initial condition $\xcont(0) = x \in \R^{n}$ and input $\ucont(\tint) = u \in \R^{m}$ for all $\tint \in \R$.
\begin{lemma}\label{lemma:single-int}
    For all 
    $\ratio > 0$, there exist $\bprsi, \lip_\slow, \lip_G, \tilde{\lip}_{\slow} > 0$ such that, for all $\pr \in (0, \bprsi)$ and $\prct = \ratio\pr$, it holds
    \begin{subequations}
        \begin{align}
            \norm{\slow(x,\xi,\Pi(z),\pr) - \slow(x,\xi^\prime, \Pi(z^\prime), \pr)} &\leq \pr\lip_\slow\norm{\begin{bmatrix} 
                \xi - \xi^\prime
                \\
                z - z^\prime
            \end{bmatrix}}
            \label{eq:slow_condition}
            \\
            \norm{\begin{bmatrix} 
                \fast(x,\xi,\Pi(z),\pr) - \fast(x,\xi^\prime,\Pi(z^\prime),\pr)
                \\
                \mpc(z,x) - \mpc(z^\prime,x)
            \end{bmatrix}}
             &\leq \lip_{\Fast}\norm{\begin{bmatrix} 
                \xi - \xi^\prime
                \\
                z - z^\prime
            \end{bmatrix}}
            \label{eq:fast_condition}
            \\
            \norm{\red(\x, \Pi(\zstar(\x)), \pr) - \x} &\leq \pr\tilde{\lip}_{\slow}\norm{x - \xstar},
            \label{eq:red_condition}
        \end{align}
    \end{subequations}
    for all $x \in \feas$, $\xi, \xi^\prime \in \domXi$, and $z, z^\prime \in \domZ$. \oprocend
\end{lemma}
The proof of Lemma~\ref{lemma:single-int} is provided in Appendix~\ref{sec:proof_lemma_single_int}.

\subsection{Boundary-Layer System Analysis}
\label{sec:bl}
We now focus on the boundary-layer system, i.e., the dynamics obtained by considering an arbitrarily fixed slow state $\xt = x$ for all $\iter \in \N$ in the fast subsystem~\eqref{eq:interconnected_system_xi}-\eqref{eq:interconnected_system_z}.
A graphical representation is shown in Fig.~\ref{fig:bd_boundary_layer}.
\begin{figure}[H]
  \centering
  \includegraphics[scale=\scale]{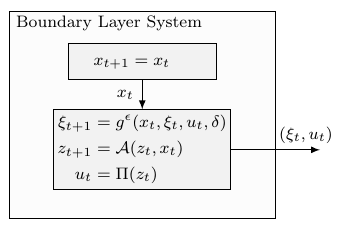}
  \caption{Block diagram representation of the boundary layer system.}
  \label{fig:bd_boundary_layer}
\end{figure}
In order to describe this system, we introduce a set of novel coordinates $(x, \txi, \tz)$ defined as
\begin{align}
    \begin{bmatrix} 
        x 
        \\
        \xi  
        \\
        \z 
    \end{bmatrix} 
    \longmapsto 
    \begin{bmatrix}
        x 
        \\
        \txi
        \\
        \tz
    \end{bmatrix} 
    := 
    \begin{bmatrix} 
        x 
        \\
        \xi - \xieq(x, \Pi(\z))
        \\
        \z - \zstar(x)
    \end{bmatrix}.\label{eq:error_coordinates}
\end{align}

Then, by observing~\eqref{eq:interconnected_system_xi}-\eqref{eq:interconnected_system_z} and recalling that $\xt = x$ for all $\iter \in\N$, the boundary-layer system 
reads as
\begin{subequations}\label{eq:BL}
    \begin{align}
        \txitp &= \tfast(x, \txit, \tzt,\pr) + \Delta\xieq(x, x, \tzt,\tztp)
        \\
        \tztp &= \tmpc(\tzt, x),
    \end{align}
\end{subequations}
where, for the sake of compactness, we introduce %
\begin{subequations}\label{eq:tilde_definitions}
    \begin{align}
        \tfast(x, \txi, \tz, \pr) 
        &:=  - \xieq(x, \Pi(\tz + \zstar(x)))
        +\fast(x, \txi  +  \xieq(x, \Pi(\tz + \zstar(x))), \Pi(\tz  +  \zstar(x)),\pr)
        \label{eq:tfast}
        \\
        \tmpc(\tz, x) &:= \mpc(\tz + \zstar(x), x) - \zstar(x)\label{eq:tmpc}
        \\
        \Delta \xieq(x,x^\prime, \tz,\tz^\prime) &:= \xieq(x, \Pi(\tz + \zstar(x))) 
        %
        - \xieq(x^\prime, \Pi(\tz^\prime + \zstar(x^\prime))).
        \label{eq:delta_xieq}
    \end{align}
\end{subequations}
To formally state global exponential stability of the origin for the boundary-layer system~\eqref{eq:BL}, let us introduce %
\begin{subequations}
    \begin{align}
        U(\txi, \tz) &:= \cG(\txi) + \kappa\cL(\tz)\label{eq:cU_def}
        \\
        \tz\ud{+} &:= \tmpc(\tz, x)
        \label{eq:tz_plus}
        \\
        \txi\ud{+} &:= \tfast(x, \txi, \tz, \pr) + \Delta\xieq(x, x, \tz,\tz\ud{+}),
        \label{eq:txi_plus}
    \end{align}
\end{subequations}
where $\cG$ and $\cL$ are introduced in~\eqref{eq:U} (cf. Assumption~\ref{ass:ges_fast}) and~\eqref{eq:cL} (cf. Assumption~\ref{ass:ges_MPC}), respectively, while $\kappa > 0$ is a constant that will be characterized in the formal result.
\begin{lemma}\label{lemma:bl}
    For all $\bar{\pr}_2, \ratio > 0$, there exist $\bar{\kappa} > 0$ such that, for all $\kappa > \bar{\kappa}$, $\pr \in (0, \bar{\pr}_2)$, and $\prct = \ratio\pr$, it holds
    \begin{subequations}\label{eq:cU}
        \begin{align}
        &b_1 (\snorm{\txi}^2 + \snorm{\tz}^2) \leq U(\txi, \tz) \leq b_2 (\snorm{\txi}^2  +  \snorm{\tz}^2)
        \label{eq:cU_1}
        \\
        &U\big(\txi\ud{+},\tz\ud{+}\big) - U(\txi,\tz) \leq -b_3 (\snorm{\txi}^2 + \snorm{\tz}^2),\label{eq:cU_2}
        \end{align}
    \end{subequations}
    for all $(x,\txi,\tz) \in \feas \times \R^{p} \times \R^{\dimz}$, such that $\txi +\xieq(\x, \Pi(\tz + \zstar(x))) \in \domXi$,  $\tz + \zstar(x) \in \domZ$ and %
 some $b_1, b_2, b_3> 0$. 
    \oprocend
\end{lemma}
The proof of Lemma~\ref{lemma:bl} is provided in Appendix~\ref{sec:bl_proof}.

\subsection{Reduced System Analysis}
\label{sec:reduced}

As graphically depicted in Fig.~\ref{fig:bd_reduced_system}, the reduced system corresponds to the slow subsystem~\eqref{eq:interconnected_system_x} studied by considering the fast state $(\xit, \zt)$ in its equilibrium manifold, i.e., $\xit = \xieq(\xt, \Pi(\zt))$ and $\zt = \zstar(\xt)$ for all $\iter \in \N$.
\begin{figure}[H]
  \centering
  \includegraphics[scale=\scale]{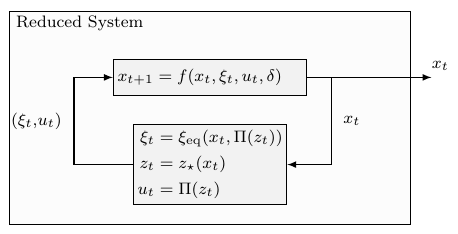}
  \caption{Block diagram representation of the reduced system.}
  \label{fig:bd_reduced_system}
\end{figure}
Then, the reduced system can be compactly written as
\begin{align}\label{eq:redcuced_system}
    \xtp = \red(\xt, \Pi(\zstar(\xt)), \pr),
\end{align}
in which we recall that $\red$ is defined in~\eqref{eq:red}.
where $\flow_{\redc}(x,u,\tint)$ is the solution at time $\tint$ to $\dot{x}(\tint) = \redc(x(\tint), u)$ initialized at $\flow_{\redc}(x,u,0) = x$ and subject to the constant input $u(\tint) = u$ for all $\tint \in [0, \pr)$.
As we formally show in the next lemma, the Lyapunov function $W$ characterized in Assumption~\ref{ass:red_ges} can be suitably used to establish exponential stability of $\xstar$ for system~\eqref{eq:redcuced_system}.
\begin{lemma}\label{lem:rs}
    For all $\gamma \in (0,\gammam)$ and $\tilde{c}_3 \in (0,c_3)$, there exists $\bar{\pr}_3 > 0$ such that, for all $\pr \in (0,\bar{\pr}_3)$, it holds 
    \begin{align*}
        W(\red(\x, \Pi(\zstar(\x)), \pr)) - W(x) \leq  - \pr \tilde{c}_3\norm{x - \xstar}^2, 
    \end{align*}
    for all $x \in \Omega_W(\gamma)$. \oprocend
\end{lemma}
The proof of Lemma~\ref{lem:rs} is provided in Appendix~\ref{sec:proof_lemma_rs}.

\subsection{Proof of Theorem~\ref{th:main}}
\label{sec:proof}

The proof of Theorem~\ref{th:main} merges the results achieved in Sections~\ref{sec:bl} and~\ref{sec:reduced} to invoke the generic timescale separation result formalized in Theorem~\ref{th:theorem_generic} (cf. Appendix~\ref{sec:generic}).
To apply it, we must verify that its requirements are satisfied.
First, the Lipschitz continuity of the equilibrium map (cf.~\eqref{eq:lipschitz_eq}) is implied by the Lipschitz continuity of $\xieq$, $\Pi$, and $\zstar$ stated in Assumptions~\ref{ass:ges_fast} and~\ref{ass:ges_MPC}. 
Further, by Lemma~\ref{lemma:single-int}, for all $\ratio > 0$, there exists $\bar{\pr}_1 > 0$ such that the conditions required in~\eqref{eq:lipschitz} are ensured for all $\pr \in (0, \bar{\pr}_1)$ and $\prct = \ratio\pr$.
Regarding the stability requirements, conditions in~\eqref{eq:U_generic} demand a Lyapunov function proving the exponential stability of the origin for the boundary-layer system \eqref{eq:BL}.
By Lemma~\ref{lemma:bl}, this requirement is satisfied by the function $U$ defined in~\eqref{eq:cU_def} 
for all $\ratio, \bar{\pr}_2 > 0$, provided that $\pr \in (0,\bar{\pr}_2)$, $\prct = \ratio\pr$, and $\kappa \geq \bar{\kappa}$.
Moreover, in light of Lemma~\ref{lem:rs}, there exists $\bar{\pr}_3 > 0$ such that, for all $\pr \in (0,\bar{\pr}_3)$, the reduced-system conditions in~\eqref{eq:W_generic} are satisfied in $\feas$ by the function $W$ characterized in Assumption~\ref{ass:red_ges}. %
Finally, we need to guarantee the condition~\eqref{eq:one_step}.
To this end, 
let us arbitrarily choose $\gamma \in (0, \gammam)$.
    Then, for all $(x,\xi,z) \in \Omega_V(\gamma)$, we can add and subtract $\red(x,\Pi(\zstar(x)),\pr) = \slow(x,\xieq(x,\Pi(\zstar(x))),\Pi(\zstar(x)),\pr)$ (cf.~\eqref{eq:red}) in the term $\norm{\slow(x,\xi,\Pi(z),\pr) - x}$, and use the triangle inequality to write 
    \begin{align}
        \norm{\slow(x,\xi,\Pi(z),\pr) - x} 
        &\leq
            \norm{\red(x,\Pi(\zstar(x)),\pr) - x}
            + \norm{\slow(x,\xi,\Pi(z),\pr) - \slow(x,\xieq(x,\Pi(\zstar(x))),\Pi(\zstar(x)),\pr)}
        \notag\\
        &\stackrel{(a)}{\leq} \pr\tilde{\lip}_{\slow}\norm{x - \xstar} +
        \pr\lip_\slow\norm{\col(
                \xi - \xieq(x, \Pi(z)),
                z - \zstar(x))}
        \notag\\
        &\stackrel{(b)}{\leq} 
        \pr\lip_\slow\sqrt{\gamma/\min\{c_1,a_1, \kappa d_1\}},
        \label{eq:new_bound}
    \end{align}
    where in $(a)$ we impose $\pr \in (0, \bprsi)$ to use~\eqref{eq:slow_condition} and~\eqref{eq:red_condition} from Lemma~\ref{lemma:single-int}, while $(b)$ follows from the bounds in~\eqref{eq:quadratic_bounds_mpc},~\eqref{eq:U_1}, and~\eqref{eq:cL_1}.
    Now, let us arbitrarily choose $\bar{\gamma} \in (\gamma, \gammam)$ and define 
    \begin{subequations}
        \begin{align}
            \bar{\Omega}_W(\bar{\gamma}) &:= \{x \in \feas \mid W(x) > \bar{\gamma}\}
            \\
            \label{eq:rho_max}
            \rho_{{\text{max}}}  &:= 
            \inf_{
                \substack{x \in \Omega_W(\gamma)
                \\
                \bar{x} \in \bar{\Omega}_W(\bar{\gamma})}}
                \norm{x - \bar{x}}.
            \end{align}
    \end{subequations}
    Then, we arbitrarily choose $\rho \in (0, \rho_{{\text{max}}})$, define $\bar{\pr}_4 := \min\{\bprsi,\frac{\rho}{\lip_\slow}\sqrt{\min\{c_1,a_1, \kappa d_1\}/\gamma}\}$, and
    use~\eqref{eq:new_bound} to get 
    \begin{align}\label{eq:rho}
        \norm{\slow(x,\xi,\Pi(z),\pr) - x} < \rho,
    \end{align} 
    for all $\pr \in (0, \bar{\pr}_4)$ and $(x,\xi,z) \in \Omega_V(\gamma)$.
    By combining~\eqref{eq:rho} with~\eqref{eq:rho_max}, we have $\slow(x,\xi,\Pi(z),\pr) \in \Omega_W(\bar{\gamma}) \subseteq \feas$ for all $\pr \in (0, \bar{\pr}_4)$ and $(x,\xi,z) \in \Omega_V(\gamma)$.
Thus, all the requirements of Theorem~\ref{th:theorem_generic} are satisfied and we can apply Theorem~\ref{th:theorem_generic} in $\Omega_V(\gamma)$.
To state the result, we observe that the definition of $\gammam$ (cf.~\eqref{eq:gammam})
leads to
\begin{align*}
    (x,\xi,z) \in \Omega_{V}(\gamma) \implies x \in \feas.
\end{align*}
Hence, by Theorem~\ref{th:theorem_generic}, for all $\ratio > 0$ and $\gamma \in (0,\gammam)$, there exists $\bar{\pr} \in (0,\min\{\bar{\pr}_1,\bar{\pr}_2,\bar{\pr}_3\})$ such that, for all $\pr \in (0,\bar{\pr})$, $\prct = \ratio\pr$, $(x,\xi,z) \in \Omega_{V}(\gamma)$, and by denoting with $(\x\ud{+},\xi\ud{+},z\ud{+})$ the corresponding next state obtained according to dynamics~\eqref{eq:interconnected_sys}, the increment $\Delta V(x,\xi,z) := V(\x\ud{+}, \xi\ud{+},z\ud{+}) - V(x, \xi, z)$ of $V$ along the trajectories of system~\eqref{eq:interconnected_sys} is upper bounded by 
\begin{align}
    \Delta V(x,\xi, z )
    &
    \leq - \lambda\norm{x - \xstar}^2 - \lambda\snorm{\xi - \xieq(x,\Pi(z))}^2
    -\lambda\norm{z - \zstar(x)}^2,\label{eq:invariance_V}
\end{align}
for some constant $\lambda > 0$.
By combining the invariance property~\eqref{eq:invariance_V} with the definition $V$ (cf.~\eqref{eq:V_pre_thm}), 
we obtain
\begin{align}\label{eq:tV_plus}
    W(\x\ud{+}) \leq \gamma.
\end{align} 
Hence, by definition of $\gammam$ (cf.~\eqref{eq:gammam}), recursive feasibility is achieved. %
Finally, we note that~\eqref{eq:invariance_V} and~\eqref{eq:tV_plus} ensure that the origin is an exponentially stable equilibrium of~\eqref{eq:interconnected_sys} (see, e.g.,~\cite[Theorem~13.2]{chellaboina2008nonlinear}) in $\Omega_{V}(\gamma)$. %

\section{Virtual Experiments}
\label{sec:simulations}

To validate \algo/, we present virtual experiments on the PenduBot system described in Section~\ref{sec:example}. The physical parameters are reported in Table~\ref{tab:pendubot_parameters}. Note that the two links have identical mechanical parameters.
\begin{table}[H]
    \centering
    \caption{Two-Link Robot Arm Parameters.}
    \label{tab:pendubot_parameters}
    \small

    \vspace{2pt}
    
    \textbf{Mechanical Parameters} \\[1pt]
    \begin{tabular}{ccccc} 
        \toprule
        $m$ [kg] & $J$ [kg m$^2$] & $l$ [m] & $l_{c}$ [m] & $b$ [Nm s] \\
        1.0 & 0.083 & 1.0 & 0.5 & 2.0 \\
        \bottomrule
    \end{tabular}
    
    \vspace{2pt}
    
    \textbf{Electrical Parameters (Joint 1)} \\[1pt]
    \setlength{\tabcolsep}{10pt}
    \begin{tabular}{cccc} 
        \toprule
        $R$ [$\Omega$] & $L$ [H] & $K_t$ [$\frac{\text{Nm}}{\text{A}}$] & $K_e$ [$\frac{\text{Vs}}{\text{rad}}$] \\
        0.6 & $\prct$ & 0.4 & 0.4 \\
        \bottomrule
    \end{tabular}
\end{table}
\begin{figure}[htpb]
    \centering
    \includegraphics[width=0.9\columnwidth]{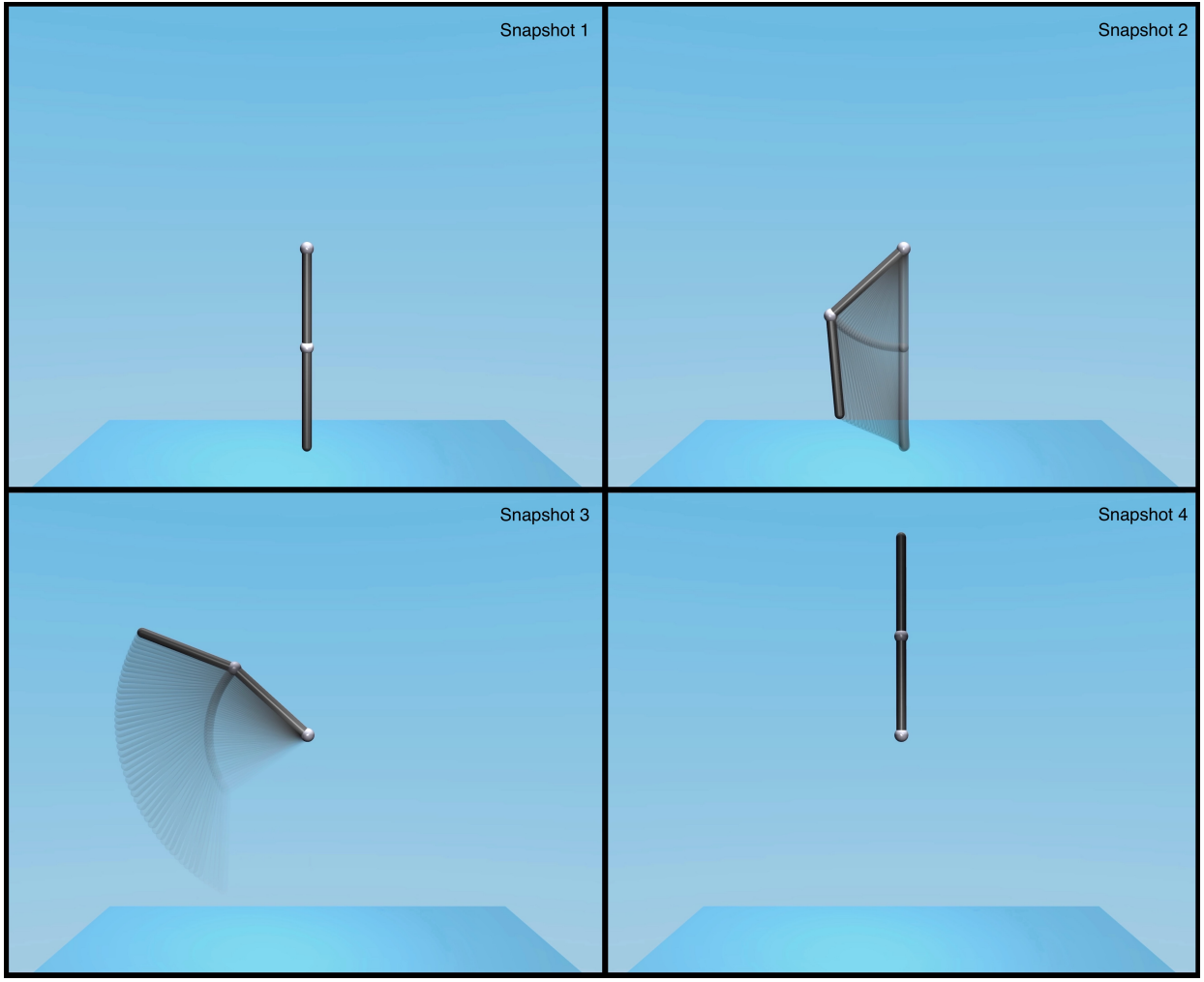}
    \caption{Sequential snapshots of the PenduBot swing-up maneuver controlled via \algo/ in MuJoCo. Previous states are shown as transparent overlays.}
    \label{fig:mujoco_snaps}
\end{figure}
Consistently with the approach outlined in the previous sections, the optimizer is only aware of the reduced-order model $\redc: \R^{4} \times \R \to \R^{4}$ obtained by considering~\eqref{eq:pendubot_dynamics_slow} with the armature current $\xi$ lying on the equilibrium manifold $\xieq(x,u) = (u - K_e x_{3}) / R$ (see~\eqref{eq:manifold_ex}).
As for the finite-horizon optimal control problem \eqref{eq:MPC_problem}, $\redd$ (namely, the discrete-time approximation of $\redc$) is obtained via a fourth-order Runge-Kutta discretization with sampling time $\prm=0.07$ s of the reduced-order model $\redc$, while the cost functions are defined as
\begin{align*}
    \cost(x,u) &= (x-x_{\text{ref}})^\top \cQ (x-x_{\text{ref}}) + (u- u_{\text{ref}})^\top \cR (u - u_{\text{ref}}), \\ 
    \costf(x) &= (x-x_{\text{ref}})^{\top} \cQ_f (x-x_{\text{ref}}),
\end{align*}
with $\cQ = \cQ_f = \diag(100, 100, 0.1, 0.1)$ and $\cR = 0.1$.
The prediction horizon is set to $\hor = \lceil \hor_s / \prm \rceil$, where $\lceil \cdot \rceil$ denotes the ceiling function, $\hor_s = 3$ s, and the input is subject to the constraint $\mathcal{U} = \{u \in \R \mid |u| \leq 24\}$. 
The MPC problem is suboptimally addressed through a real-time iteration (RTI) scheme~\cite{diehl2005real}. 
Specifically, we perform a single iteration of a Sequential Quadratic Programming (SQP) method per sampling instant.
All experiments are carried out using the MuJoCo physics engine~\cite{todorov2012mujoco}, which simulates the continuous-time plant dynamics with high-fidelity over each sampling interval. 
The PenduBot swing-up maneuver obtained in MuJoCo is illustrated in Fig.~\ref{fig:mujoco_snaps}.
The algorithmic framework is implemented in the CasADi toolkit~\cite{andersson2019casadi}. %

\subsection{Closed-Loop Performance Analysis}\label{sec:closed_loop_performance}

We evaluate the proposed approach (\emph{SMART-MPC - Reduced Prediction})
where the MPC optimizer relies on a reduced-order prediction model while interacting with the full-order plant dynamics, against two additional benchmark settings. 
In these two benchmarks, the plant dynamics are artificially reduced to perfectly match the prediction model, eliminating unmodeled effects to allow the optimizer to fully capture the system behavior.
The first benchmark solves the MPC problem suboptimally through the same RTI scheme used in our approach (\emph{Suboptimal MPC - Exact Prediction}), whereas the second computes the exact optimal solution at each iteration using a standard SQP method (\emph{Optimal MPC - Exact Prediction}).
The control task requires the PenduBot to perform a counterclockwise swing-up maneuver, starting from the downward resting position $\x_0 = [0, 0, 0, 0]^\top$ with the motor at rest $\xi_0 = 0$, and reaching the upright equilibrium configuration $\xstar = [\pi, 0, 0, 0]^\top$.
We evaluate the discussed strategies under two distinct scenarios to study how the interplay between the sampling time $\pr$ and the tuning parameter $\prct$ affects the closed-loop performance.
As established in Theorem~\ref{th:main}, a sufficiently small sampling interval and fast extra dynamics are required. 
Indeed, as shown in Fig.~\ref{fig:SQP_all_work}, where $\pr = 0.005$ s and $\prct = 0.005$, the proposed architecture effectively matches the two ideal benchmarks, ensuring excellent closed-loop performance.
In contrast, operating at a slower sampling rate of $\pr = 0.07$ s and with slower-extra dynamics ($\prct = 0.05$) (cf. Fig.~\ref{fig:SQP_one_broken}) leads to a noticeable performance degradation and oscillatory behavior.
\begin{figure}[H]
    \centering
    \includegraphics[width=0.9\columnwidth]{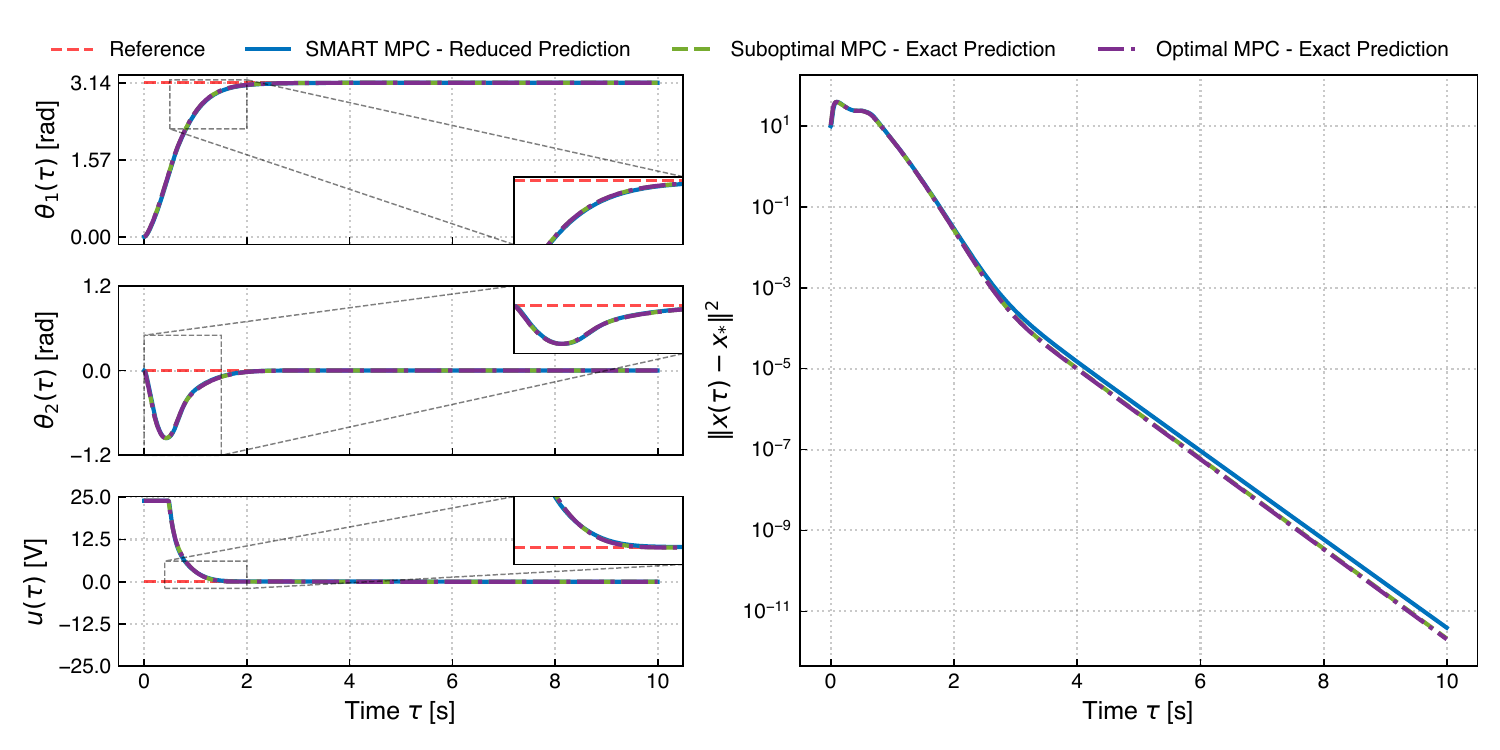}
    \caption{Closed-loop trajectories of the target states $\theta_1, \theta_2$ and control input $u$ (left) and state error squared norm in semilogarithmic scale (right) for $\pr = 0.005$ s.}
    \label{fig:SQP_all_work}
\end{figure}
\begin{figure}[H]
    \centering
    \includegraphics[width=0.9\columnwidth]{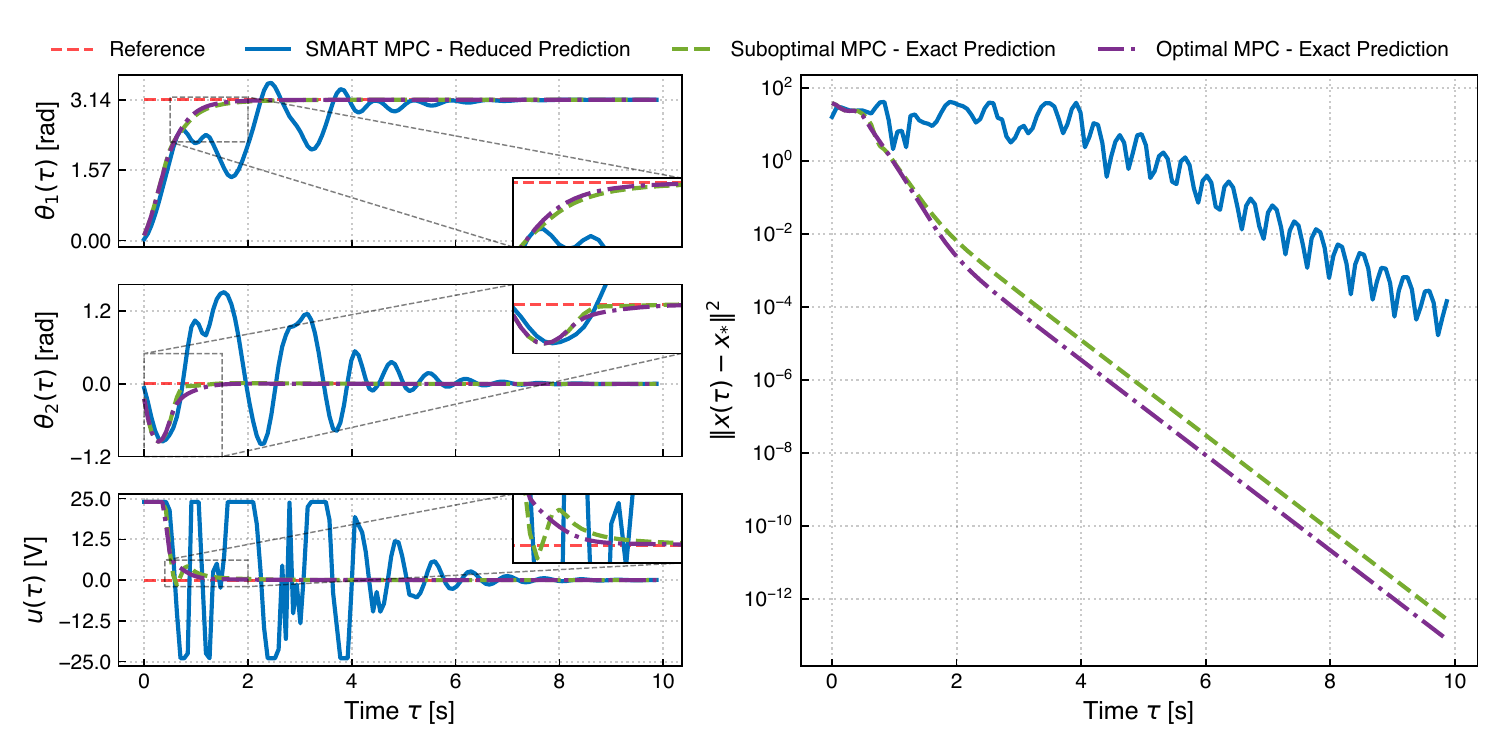}
    \caption{Closed-loop trajectories of the target states $\theta_1, \theta_2$ and control input $u$ (left) and state error squared norm in semilogarithmic scale (right) for $\pr = 0.07$ s.}
    \label{fig:SQP_one_broken}
\end{figure}
This confirms that sufficient timescale separation is necessary to utilize suboptimal MPC schemes and neglect fast unmodeled dynamics while maintaining closed-loop stability and performance guarantees.

\subsection{Tests Under Varying Conditions}

In this section, we empirically investigate the robustness of the proposed approach against two distinct sources of variability: variations in the extra subsystem dynamics and in the initial conditions. 
As in the virtual experiments presented in Section~\ref{sec:closed_loop_performance}, the PenduBot is required to reach and maintain the upright equilibrium.
First, we analyze the effect of the inductance $L$, acting, as discussed, as the tuning parameter $\prct$.
In particular, since $\prct = \ratio\pr$, by varying the value of $\ratio$, we test the controller performance as the gap between the target and the extra dynamics narrows, while keeping the sampling time $\pr$ fixed.
\begin{figure}[H]
    \centering
    \includegraphics[width=0.9\columnwidth]{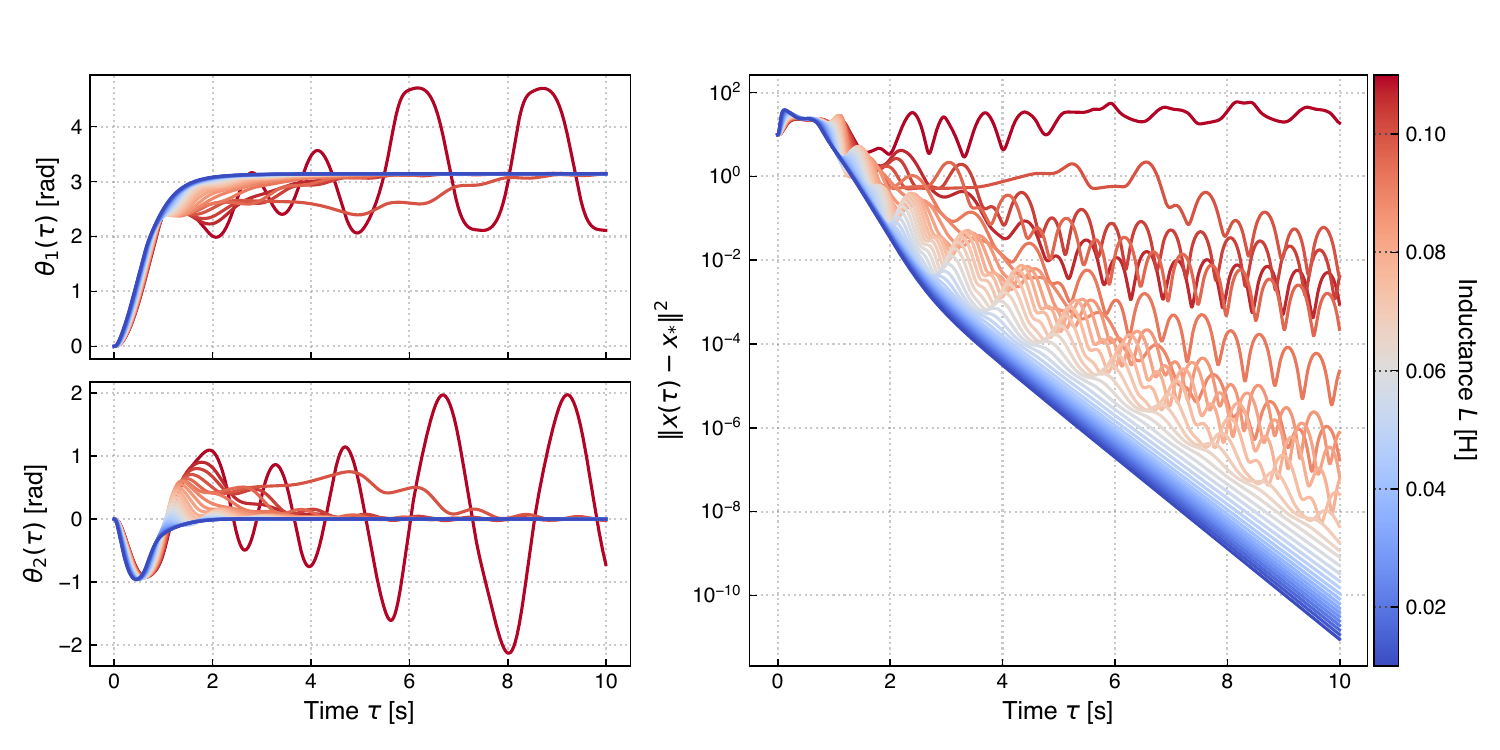}
    \caption{Closed-loop trajectories of the target states $\theta_1$, $\theta_2$ (left) and state error squared norm in semilogarithmic scale (right) for increasing values of the inductance $L$.}
    \label{fig:changing_L_trajs_errors}
\end{figure}
In Fig. \ref{fig:changing_L_trajs_errors}, the closed-loop performance of the proposed approach is evaluated for different values of $\ratio$. Specifically, 30 equally spaced values of $\ratio$ are considered ranging from $\ratio=2$ to $\ratio=22$.
The sampling time is fixed to $\pr = 0.005$ s.
This results in a range of values for the tuning parameter $\prct$ (and, thus, for the inductance $L$) spanning from $\prct = 0.010$ to $\prct = 0.110$.
The results highlight two distinct behaviors: as shown by the progressively darker blue trajectories, smaller inductance values ensure a clear timescale separation. 
In this regime, the reduced-order model provides good approximations and, thus, the controller yields fast and smooth convergence to the swing-up position.
By contrast, as the color shifts toward red with increasing values of $L$, the electrical dynamics slow down.
This leads to the emergence of unmodeled dynamics, visible as progressively pronounced oscillations in the state trajectories, resulting in a degradation of the closed-loop performance, up to failure of convergence for the largest considered values of $L$.
\\
Subsequently, we perform a Monte Carlo analysis to test the robustness of the proposed approach against variations in the initial conditions. 
Specifically, we perform $30$ trials in which we randomly initialize the system states $\theta_{1,0}$, $\theta_{2,0}$, $\omega_{1,0}$, and $\omega_{2,0}$ drawing them from a uniform distribution over the ranges $[-1.2, +1.2]$ rad, $[-0.25, +0.25]$ rad, $[-0.25, 0.25]$ rad/s, and $[-0.25, 0.25]$ rad/s, respectively.
We set $\pr = 0.005$ s, while the prediction model is obtained via a fourth-order Runge-Kutta discretization with sampling time $\prm = 0.05$.
As depicted in Fig.~\ref{fig:montecarlo_x0}, the controller successfully stabilized the PenduBot at the upright equilibrium in all trials. 
\begin{figure}[H]
    \centering
    \includegraphics[width=0.9\columnwidth]{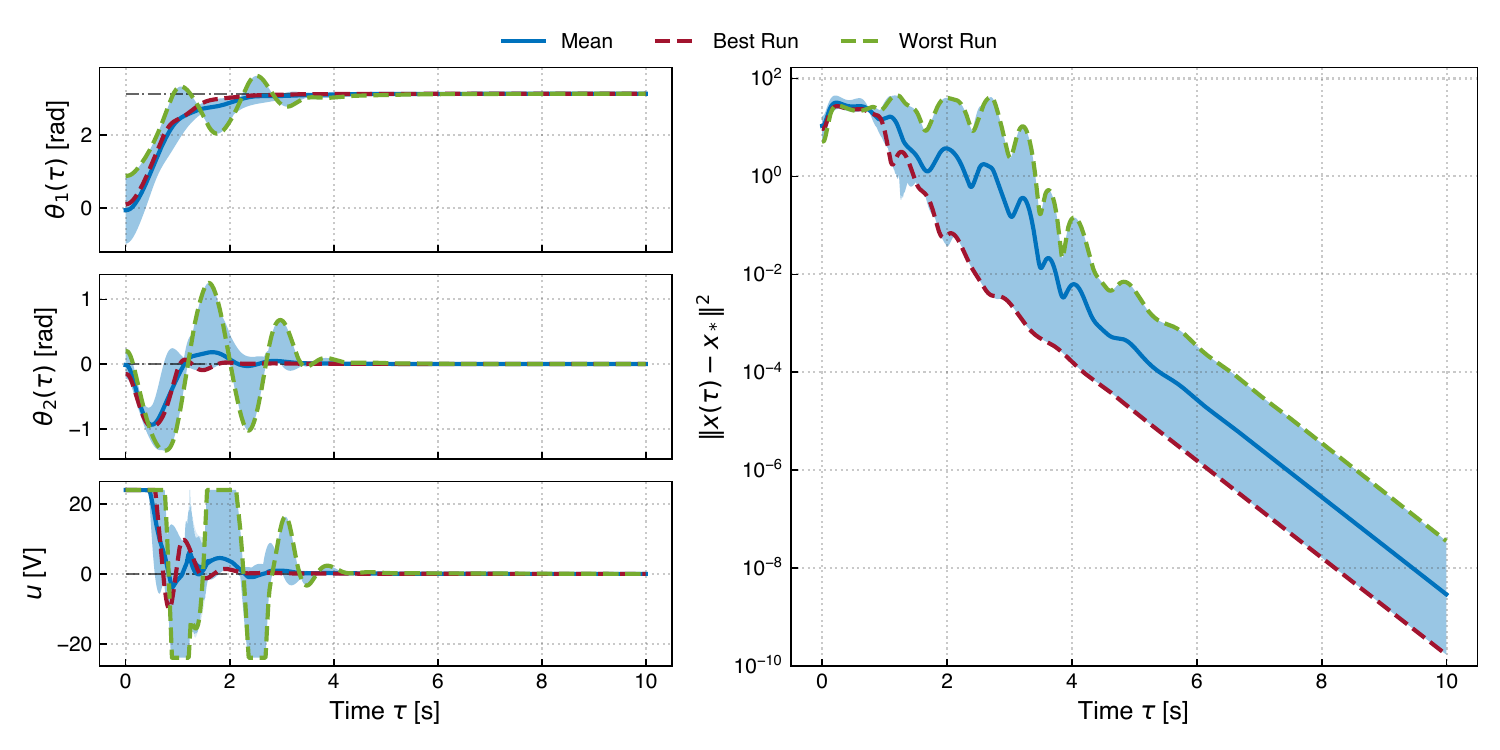}
    \caption{Closed-loop trajectories of the target states $\theta_1, \theta_2$ and control input $u$ (left) and state error squared norm in semilogarithmic scale (right) for randomly generated initial conditions. The shaded regions indicate the envelope spanned by the minimum and maximum values across all realizations while the solid lines represent the average behavior, and the best and worst-case trajectories.}
    \label{fig:montecarlo_x0}
\end{figure}

\section{Conclusions}
\label{sec:conclusions}
We proposed a suboptimal and reduced-order MPC scheme for sampled-data control of continuous-time interconnected systems.
The approach combines a suboptimal optimizer with a reduced-order prediction model, obtained by neglecting part of the plant dynamics, enabling a lighter computational burden. We provided formal guarantees on both recursive feasibility and closed-loop stability by properly tuning the plant sampling time.
Virtual experiments on an underactuated two-link robotic arm, including a Monte Carlo analysis, confirmed the effectiveness of our approach.

\appendix
\subsection{Exponential Stability of Two-Time-Scales Systems}
\label{sec:generic}

The following result provides sufficient conditions for the exponential stability of two-time-scale discrete-time systems.
It is an extension of~\cite[Thm.~II.5]{carnevale2022tracking}, since here we allow for (i) generic slow dynamics in place of the specific form $\xtp = \xt + \pr \slow(\xt, \wt)$ considered there (ii) and local stability results in place of global ones.
\begin{theorem}
    \label{th:theorem_generic}
	Consider the system
	\begin{subequations}\label{eq:interconnected_system_generic}
		\begin{align}
			\xtp &= \slow(\xt,\yt,\pr)\label{eq:slow_system_generic}
			\\
			\ytp &= \Fast(\xt, \yt,\pr),\label{eq:fast_system_generic}
		\end{align}
	\end{subequations}	
	with $\xt \in \domX \subseteq \R^n$, $\yt \in \domY \subseteq \R\du{m}$, $\map{\slow}{\domX \times \domY \times \R_+}{\domX}$, $\map{\Fast}{\domX \times \domY \times\R_+}{\domY}$, and $\pr > 0$.
    Assume that there exist $\domX_0 \subseteq \domX$, $\eq: \domX_0 \to \domY$, and $\lip_{\eq} >0$ such that
    \begin{subequations}\label{eq:generic_assumptions_equilibrium}
        \begin{align}
            \eq(\x) &= \Fast(\x, \eq(\x), \pr)\label{eq:fast_equilibrium}
            \\
            \norm{\eq(\x) - \eq(x^\prime)} &\leq \lip_{\eq} \norm{\x - \x^\prime},\label{eq:lipschitz_eq}
        \end{align}
    \end{subequations}
    for all $\x, \x^\prime   \in  \domX_0$.
    Assume also that there exists $\xstar \in \domX_0$ such that $\xstar = \slow(\xstar, \eq(\xstar),\pr)$ for all $\pr > 0$. 
    Then, let
	\begin{equation}\label{eq:redcuced_system_generic}
		\xtp = \slow(\xt,\eq(\xt),\pr)
	\end{equation}
	be the \emph{reduced system} and
	\begin{equation}\label{eq:boundary_layer_system_generic}
		\psitp = \Fast(\x, \psit + \eq(\x), \pr) - \eq(\x)
	\end{equation}
	be the \emph{boundary layer system}.

    Assume that there exist $L_{\slow}, L_{\Fast}, \bar{\pr}_1 > 0$ such that for all $\pr \in (0,\bar{\pr}_1)$, it holds
    \begin{subequations}\label{eq:lipschitz}
        \begin{align}
            \norm{\slow(\x,\y,\pr) - \slow(\x,\y^\prime,\pr)} 
            &\leq \pr L_{\slow} \norm{\y - \y^\prime} \label{eq:lipschitz_generic_slow}
            \\
            \norm{\Fast(\x,\y,\pr)-\Fast(\x,\y^\prime,\pr)} 
            &\leq L_{\Fast} \norm{\y-\y^\prime}\label{eq:lipschitz_generic_fast}
            \\
            \norm{\slow(\x,\eq(\x),\pr) - \x}
            &\leq \pr \tilde{L}_{\slow} \norm{\x - \xstar},
            \label{eq:single_integrator_like} 
        \end{align}
    \end{subequations}
    for all $\x \in \domX_0$, $\y, \y^\prime \in \domY$.
	Further, assume there exist a continuously differentiable function $U$ and $\bar{\pr}_2 > 0$ such that, for all $\pr \in (0,\bar{\pr}_2)$, there exist $b_1, b_2, b_3 > 0$ such that for all $\x \in \domX_0$, $\psi,  \psi^\prime$ such that $\psi + \eq(\x) \in \domY$, it holds
	\begin{subequations}\label{eq:U_generic}
		\begin{align}
			&b_1 \norm{ \psi}^2 \leq U( \psi) \leq b_2\norm{ \psi}^2 \label{eq:U_first_bound_generic}
			\\
			&U(\Fast(\x, \psi  + \eq(\x), \pr) - \eq(\x))  -  U( \psi)  \leq   -b_3\norm{ \psi}^2. 
            \label{eq:U_minus_generic}
		\end{align}
	\end{subequations}
    Further, $\nabla U$ is $b_4$-Lipschitz continuous for some $b_4 > 0$.
	Moreover, assume there exist a continuously differentiable function $W: \domX_0 \to \R$ and $\bar{\pr}_3 > 0$ such that, for all $\pr \in (0,\bar{\pr}_3)$, it holds
	\begin{subequations}\label{eq:W_generic}
		\begin{align}
			&c_1 \norm{\x - \xstar}^2 \leq W(\x) \leq c_2\norm{\x - \xstar}^2\label{eq:W_first_bound_generic}
			\\
			&W(\slow(\x,\eq(\x),\pr))  -  W(\x)  \leq  -\pr c_3\norm{\x - \xstar}^2,\label{eq:W_minus_generic}
		\end{align}
	\end{subequations}
    for all $x \in \domX_0$ and some $c_1, c_2, c_3 > 0$.
    Further, $\nabla W$ is $c_4$-Lipschitz continuous for some $c_4 > 0$.
    Finally, assume there exists $\bar{\pr}_4 > 0$ such that, for all $\pr \in (0,\bar{\pr}_4)$, it holds
    \begin{align}\label{eq:one_step}
        \slow(\x,y,\pr) \in \domX_0,
    \end{align}
    for all $(\x,\y) \in \domX_0 \times \domY$.
    Then, there exists $\bar{\pr} \in (0,\min\{\bar{\pr}_1,\bar{\pr}_2, \bar{\pr}_3,\bar{\pr}_4\})$ such that, for all $\pr \in (0,\bar{\pr})$, it holds
    \begin{align*}
        &W(\slow(x,y,\pr)) + U(\Fast(\x, y, \pr) - \eq(\slow(x,y,\pr)))
        - W(x) - U(\y - \eq(\x))
        \leq -\lambda
        (\norm{\x - \xstar}^2 + \norm{\y - \eq(\x)}^2),
    \end{align*}
    for all $(\x,\y) \in \domX_0 \times \domY$.
\end{theorem}
\begin{proof}

    First, since it will be useful later, we note that the Lipschitz continuity of $\nabla U$ and $\nabla W$ implies that 
    \begin{subequations}
        \begin{align}
            |U( \psi)-U( \psi^\prime)|&\leq \tfrac{b_4}{2}\norm{\psi- \psi^\prime}(\norm{\psi} + \norm{\psi^\prime})
            \label{eq:U_bound_generic}
            \\
            |W(\x)-W(\x^\prime)|&\leq \tfrac{c_4}{2}\norm{\x - \x^\prime}(\norm{\x - \xstar} + \norm{\x^\prime - \xstar}),\label{eq:W_bound_generic}
        \end{align}
    \end{subequations}
    for all $\x, \x^\prime \in \domX$ and $\psi, \psi^\prime$ such that $\psi + \eq(\x) \in \domY$ and $\psi^\prime + \eq(\x^\prime) \in \domY$.
    Now, we introduce $\tyt := \yt - \eq(\xt)$, impose $\pr \in (0,\bar{\pr}_4) \implies \xtp \in \domX_0$, and rewrite~\eqref{eq:interconnected_system_generic} as 
    \begin{subequations}\label{eq:interconnected_system_generic_err}
        \begin{align}
            \xtp &=  \slow(\xt,\tyt + \eq(\xt),\pr)\label{eq:slow_system_generic_err}
            \\
            \tytp &= \Fast(\xt, \tyt +  \eq(\xt), \pr) - \eq(\xtp).\label{eq:fast_system_generic_err}
        \end{align}
    \end{subequations}%
    Now, let us consider the functions $W$ and $U$ characterized in~\eqref{eq:W_generic} and~\eqref{eq:U_generic}, respectively, and a pair $(x,\ty) \in \{(x,\ty) \in \domX_0 \times \R^{m} \mid \ty + \eq(\x) \in \domY\}$.
    By evaluating the increment $\Delta W(\x,\ty) := W(\slow(\x,\ty+\eq(\x),\pr)) - W(\x)$ of $W$ along the trajectories of subsystem~\eqref{eq:slow_system_generic_err}, we obtain 
    \begin{align}
        \Delta W(\x,\ty) 
        = W(\slow(\x,\ty+\eq(\x),\pr)) - W(\x)
        &\stackrel{(a)}{=} 
        W(\slow(\x,\eq(\x),\pr)) - W(\x) 
        + W(\slow(\x,\ty + \eq(\x),\pr)) - W(\slow(\x,\eq(\x),\pr))
        \notag\\
        &\stackrel{(b)}{\leq}
         -\pr c_3\norm{\x - \xstar}^2 + W(\slow(\x,\ty + \eq(\x),\pr)) 
        - W(\slow(\x,\eq(\x),\pr))
        \notag\\
        &\stackrel{(c)}{\leq}
        -\pr c_3\norm{\x - \xstar}^2 
        + \pr \frac{c_4}{2} L_{\slow}\norm{\ty}\norm{\slow(\x,\ty + \eq(\x),\pr) - \xstar}
        \notag\\
        &\hspace{.4cm}
        +
        \pr \frac{c_4}{2}L_{\slow}\norm{\ty}\norm{\slow(\x, \eq(\x),\pr) - \xstar},       
        \label{eq:deltaU_generic}
    \end{align}
    where in $(a)$ we add $\pm W(\slow(\x,\eq(\x),\pr))$, in $(b)$ we impose $\pr \in (0,\bar{\pr}_3)$ and we use~\eqref{eq:W_minus_generic} to bound the first two terms, while in $(c)$ we use the bound~\eqref{eq:W_bound_generic} and the Lipschitz continuity of $\slow$ (cf.~\eqref{eq:lipschitz_generic_slow}).
    By adding and subtracting $\x$ in $\norm{\slow(\x, \eq(\x),\pr) - \xstar}$ and $f(\x, \eq(\x), \pr)$ in $\norm{\slow(\x,\ty + \eq(\x),\pr) - \xstar}$, using the triangle inequality, imposing $\pr \in (0,\bar{\pr}_1)$, and applying~\eqref{eq:lipschitz_generic_slow},~\eqref{eq:single_integrator_like}, we get
    \begin{subequations}
        \begin{align} 
            &\|\slow(\x, \ty + \eq(\x),\pr) - \xstar\|
            \leq
            (1  +  \pr \tilde{L}_{\slow})\norm{\x - \xstar}  +  \pr L_{\slow}\norm{\ty}
            .
            \label{eq:bound_Lf_L_{\eq}_wt}
            \\
            &\norm{\slow(\x, \eq(\x),\pr) - \xstar} \leq  (1+ \pr \tilde{L}_{\slow})\norm{\x - \xstar}.\label{eq:bound_L_{\slow}_L_{\eq}}
        \end{align}
    \end{subequations}
    Hence, by using the inequalities~\eqref{eq:bound_Lf_L_{\eq}_wt} and~\eqref{eq:bound_L_{\slow}_L_{\eq}}, we then further bound~\eqref{eq:deltaU_generic} as 
    \begin{align}
        \Delta W (\x,\ty) 
        &\leq
        -\pr c_3\norm{\x - \xstar}^2 + \pr^2 k_3\norm{\ty}^2
        + (\pr k_1 + \pr^2 k_2)\norm{\ty}\norm{\x - \xstar}, \label{eq:deltaU_generic_final} 
    \end{align}
    introducing $k_1 := c_4\tilde{L}_{\slow}, k_2 := c_4 L_{\slow}\tilde{L}_{\slow}, k_3 := \frac{c_4}{2} L_{\slow}^2$.
    Now, let us introduce the function $\Delta \eq$ defined as 
    \begin{align}
        \Delta \eq(\x,\ty,\pr) := - \eq(\slow(\x,\ty + \eq(\x),\pr)) + \eq(\x).\label{eq:Delta_eq}
    \end{align}
    Then, we select $U$ as in~\eqref{eq:U_generic} and still consider the same pair $(x,\ty)$ considered above.
    By adding and subtracting $U(\Fast(\ty + \eq(\x),\x,\pr)-\eq(\x))$ to the increment $\Delta U(\ty,\x) := U(\Fast(\ty + \eq(\x),\x,\pr) - \eq(\slow(\x,\ty+\eq(\x),\pr))) - U(\ty)$ of $U$ along the trajectories of subsystem~\eqref{eq:fast_system_generic_err}, imposing $\pr \in (0,\bar{\pr}_2)$, and exploiting~\eqref{eq:U_minus_generic}, we get
    \begin{align}
        \Delta U(\ty,\x) 
        &\stackrel{(a)}{\leq}  -b_3\norm{\ty}^2  + \tfrac{b_4}{2}\norm{\Delta \eq(\x,\ty,\pr)}^2
        \label{eq:DeltaW_generic}
        + b_4\norm{\Delta \eq(\x,\ty,\pr)}\norm{\Fast(\x, \ty  +  \eq(\x),\pr) -  \eq(\x)},
    \end{align}
    where in $(a)$ we use~\eqref{eq:U_bound_generic} and the triangle inequality. 
    By using the definition of $\Delta \eq$ (cf.~\eqref{eq:Delta_eq}) and the Lipschitz continuity of $\eq$ in $\domX_0$ (cf.~\eqref{eq:lipschitz_eq}), for all $\pr \in (0,\min\{\bar{\pr}_1,\bar{\pr}_2,\bar{\pr}_3\})$, we can write
    \begin{align}
        \norm{\Delta \eq(\x,\ty,\pr)}
        %
        \leq L_{\eq}\norm{\slow(\x,\ty+\eq(\x),\pr) - \x} 
        &\stackrel{(a)}{\leq}
        L_{\eq}\norm{\slow(\x,\eq(\x),\pr) - \x} 
        + L_{\eq}\norm{\slow(\x,\ty+\eq(\x),\pr) - \slow(\x,\eq(\x),\pr)}
        \notag\\
        &\stackrel{(b)}{\leq}
        \pr L_{\eq}\tilde{L}_{\slow}\norm{\x - \xstar} + \pr L_{\eq}L_{\slow}\norm{\ty}
        ,
        \label{eq:bound_Deltah}
    \end{align}
    where in $(a)$ we add $\pm\slow(\x,\eq(\x),\pr)$ in the norm and apply the triangle inequality, while in $(b)$ we apply~\eqref{eq:single_integrator_like} to bound the first term and~\eqref{eq:lipschitz_generic_slow} to bound the second one.
    Further, since $\Fast(\x,\eq(\x),\pr) = \eq(\x)$ (cf.~\eqref{eq:fast_equilibrium}), we get
    \begin{align}
        &\norm{\Fast(\x, \ty + \eq(\x),\pr) - \eq(\x)} 
        =\norm{\Fast(\x, \ty + \eq(\x),\pr) - \Fast(\x, \eq(\x),\pr)}
         \leq 
        L_{\Fast}\norm{\ty},\label{eq:bound_g}
    \end{align}
    where the inequality is due to the Lipschitz continuity of $\Fast$ (cf. \eqref{eq:lipschitz_generic_fast}). 
    Using inequalities~\eqref{eq:bound_Deltah} and~\eqref{eq:bound_g}, for all $\pr \in (0,\min\{\bar{\pr}_1,\bar{\pr}_2,\bar{\pr}_3,\bar{\pr}_4\})$, we can bound~\eqref{eq:DeltaW_generic} as 
    \begin{align}
        \Delta U(\ty,\x) 
        &\leq
        (-b_3 +\pr k_6 + \pr^2 k_7)\norm{\ty}^2 + \pr^2 k_8\norm{\x - \xstar}^2
        + (\pr k_4 + \pr^2 k_5)\norm{\x - \xstar}\norm{\ty}, \label{eq:deltaW_generic_final}
    \end{align}
    where we introduce the constants 
    \begin{align*}
        k_4  :=  b_4 L_{\eq} L_{\Fast}\tilde{L}_{\slow}, \quad
        k_5 := b_4 L_{\eq}^2 \tilde{L}_{\slow}L_\slow,
        \quad
        k_6 := b_4 L_{\eq} L_{\Fast}L_{\slow}, \quad
        k_7 := \frac{b_4}{2} L_{\eq}^2 L_{\slow}^2, \quad
        k_8 := \frac{b_4}{2} L_{\eq}^2 \tilde{L}_{\slow}^2.
    \end{align*}
    Now, we use $W$ and $U$ to introduce a candidate Lyapunov function for the overall system~\eqref{eq:interconnected_system_generic_err}, namely
    \begin{align*}
        V(\x,\ty) := W(\x) + U(\ty). 
    \end{align*}
    Then, we evaluate the increment $\Delta V(\x,\ty) := V(\slow(\x,\ty+\eq(\x),\pr),\Fast(\x,\ty+\eq(\x),\pr) - \eq(\slow(\x,\ty+\eq(\x),\pr))) - V(\x,\ty) = \Delta W(\x,\ty) + \Delta U(\ty,\x)$ of $V$ along the trajectories of system~\eqref{eq:interconnected_system_generic_err}.
    By~\eqref{eq:deltaU_generic_final} and~\eqref{eq:deltaW_generic_final}, it holds
    \begin{align}
        \Delta V(\x,\ty) &\leq 
        -
        \begin{bmatrix}
            \norm{\x - \xstar}
            \\
            \norm{\ty}
        \end{bmatrix}\T Q(\pr) 
        \begin{bmatrix}
            \norm{\x - \xstar}
            \\
            \norm{\ty}
        \end{bmatrix},\label{eq:deltaV_generic}
    \end{align}
    where we introduce $Q(\pr) = Q(\pr)\T \in \R^2$ defined as 
    \begin{align*}
        &Q(\pr) := \begin{bmatrix}
            \pr c_3 -\pr^2k_8& q_{21}(\pr)\\
            q_{21}(\pr) & b_3 - \pr k_6 - \pr^2 (k_3 + k_7)
        \end{bmatrix},
    \end{align*}
    with $q_{21}(\pr) := -\frac{1}{2}(\pr (k_1+k_4) + \pr^2 (k_2+k_5))$. 
    By Sylvester criterion, we know that $Q \succ 0$ if and only if 
        \begin{align}\label{eq:conditions_generic}
            \begin{cases} 
                \pr c_3 - \pr^2 k_8 &> 0 
                \\
                c_3 b_3 &> \frac{p(\pr)}{\pr},
            \end{cases} 
        \end{align}
    where the polynomial $p(\pr)$ is defined as 
    \begin{align}
        p(\pr) &:= q_{21}(\pr)^2 + \pr^2 c_3k_6 
        + \pr^2 (\pr c_3(k_3 + k_7)+b_3k_8)
         -\pr^3k_6k_8 - \pr^4k_8(k_3+k_7).
    \end{align}
    Then, to achieve the first condition of~\eqref{eq:conditions_generic} it follows that $\pr < c_3/k_8=:\bar{\pr}_5$. 
    Further, since $p(\pr)$ is a continuous function and $\lim_{\pr \to 0}p(\pr)/\pr = 0$,
    there exists $\bar{\pr} \in (0,\min\{\bar{\pr}_1,\bar{\pr}_2, \bar{\pr}_3, \bar{\pr}_4, \bar{\pr}_5\})$ 
    such that also the second condition of~\eqref{eq:conditions_generic} is satisfied for all $\pr \in (0,\bar{\pr})$.
    Under this choice of $\pr$, the proof follows by denoting with $\lambda > 0$ the corresponding smallest eigenvalue of $Q(\pr)$.
\end{proof}

\subsection{Proof of Lemma~\ref{lemma:single-int}}
\label{sec:proof_lemma_single_int}

We assume $\pr \in (0,\bprsi)$, where $\bprsi > 0$ will be characterized later.
By definition of $\slow$ (cf.~\eqref{eq:slow_definition}) and denoting with $(\xcont(\tint),\xicont(\tint)), (\xcont^\prime(\tint),\xicont^\prime(\tint)) \in \domX \times \domXi$ the solution of~\eqref{eq:plant_ct} at time $\tint$ with initial conditions $(x,\xi), (x,\xi^\prime) \in \domX \times \domXi$ and constant inputs $u, u^\prime \in \domU$, respectively, it holds 
\begin{align}
    \norm{\slow(x,\xi,u,\pr) - \slow(x,\xi^\prime,u^\prime,\pr)}
    &=\int_{0}^\pr \norm{\slowc(\xcont(\tint),\xicont(\tint),u) - \slowc(\xcont^\prime(\tint),\xicont^\prime(\tint),u^\prime)}d\tint,
\notag\\
&\stackrel{(a)}{\leq}
    \pr\lip_{\slowc}\norm{u  -  u^\prime}  +  \lip_{\slowc}\int_{0}^\pr\norm{\begin{bmatrix}\xcont(\tint) - \xcont^\prime(\tint) \\ \xicont(\tint) - \xicont^\prime(\tint)\end{bmatrix}} d\tint,
    \label{eq:second_bound}
\end{align}
where $(a)$ uses the Lipschitz continuity of $\slowc$ (cf. Assumption~\ref{ass:slowness}) and the triangle inequality.
Now, we note that
\begin{align}
    \frac{d }{d\tint} \norm{\begin{bmatrix}\xcont(\tint) - \xcont^\prime(\tint) \\ \xicont(\tint) - \xicont^\prime(\tint)\end{bmatrix}}
    &\leq
    \norm{\begin{bmatrix}\slowc(\xcont(\tint),\xicont(\tint),u) - \slowc(\xcont^\prime(\tint),\xicont^\prime(\tint),u^\prime)
        \\ 
        \frac{1}{\prct}\fastc(\xcont(\tint),\xicont(\tint),u) - \frac{1}{\prct}\fastc(\xcont^\prime(\tint),\xicont^\prime(\tint),u^\prime)\end{bmatrix}}
        \notag\\
        &\stackrel{(a)}{\leq}
        \sqrt{\lip_{\slowc}^2 + \lip_{\fastc}^2/\prct^2}
        \norm{\begin{bmatrix}\xcont(\tint) - \xcont^\prime(\tint) \\ \xicont(\tint) - \xicont^\prime(\tint)
            \\\end{bmatrix}} 
            + 
            \sqrt{\lip_{\slowc}^2 + \lip_{\fastc}^2/\prct^2}
            \norm{u - u^\prime},
            \label{eq:third_bound}
        \end{align}
where in $(a)$ we use the Lipschitz continuity of $\slowc$ and $\fastc$ (cf. Assumption~\ref{ass:slowness}) and the triangle inequality.
By applying Gronwall's inequality (see, e.g., \cite[Th. 1.3.1]{ames1997inequalities}) to~\eqref{eq:third_bound} and defining $\alpha_\prct := \sqrt{\lip_{\slowc}^2 + \lip_{\fastc}^2/\prct^2}$,
we get
\begin{align}
    \norm{
        \begin{bmatrix}
            \xcont(\tint) - \xcont^\prime(\tint) 
            \\
            \xicont(\tint) - \xicont^\prime(\tint)
        \end{bmatrix}} 
        &
        \leq 
        \exp\left(\alpha_\prct \tint\right)
        \norm{
            \begin{bmatrix}
            \xcont(0) - \xcont^\prime(0) 
            \\
            \xicont(0) - \xicont^\prime(0)
        \end{bmatrix}
        } 
        + \norm{u - u^\prime}\left(\exp\left(\alpha_\prct \tint\right) - 1\right)
        \notag\\
        &
        \stackrel{(a)}{\leq}
        \exp\left(\alpha_\prct \tint\right)
        \norm{\xi - \xi^\prime
        } 
        + \norm{u - u^\prime}\left(\exp\left(\alpha_\prct \tint\right) - 1\right),
        \label{eq:fourth_bound}
\end{align}
where in $(a)$ we use (i) $\xcont(0) = \xcont^\prime(0)$ and (ii) $\xicont(0) = \xi$ and $\xicont^\prime(0) = \xi^\prime$.
By plugging~\eqref{eq:fourth_bound} into~\eqref{eq:second_bound}, we get
\begin{align}
    &\norm{\slow(x,\xi,u,\pr) - \slow(x,\xi^\prime,u^\prime,\pr)}
    \notag\\
    &\leq
    \lip_{\slowc}\pr\norm{u - u^\prime}
    + \lip_{\slowc}\norm{\xi - \xi^\prime}\int_{0}^\pr \exp\left(\alpha_\prct \tint      \right) d\tint
    +\lip_{\slowc}\norm{u - u^\prime}\int_{0}^\pr \left(\exp\left(\alpha_\prct \tint\right) - 1\right) d\tint
    \notag\\
    &\stackrel{(a)}{=}
    \lip_{\slowc}\left(\tfrac{1}{\alpha_\prct}\left(\exp\left(\alpha_\prct \pr\right) - 1\right)\right) \left(
        \norm{\xi - \xi^\prime} + \norm{u - u^\prime}
    \right)
    \notag\\
    &\stackrel{(b)}{\leq}
    \pr\lip_{\slowc}\exp(\alpha_\prct \bprsi) \left(
        \norm{\xi - \xi^\prime} + \norm{u - u^\prime}
    \right)
    \notag\\
    &\stackrel{(c)}{\leq}
    \pr\lip_{\slowc}
    \exp\left(\sqrt{\lip_{\slowc}^2\bar{\pr}_1^2 + \lip_{\fastc}^2/\ratio^2}\right)
        (\norm{\xi - \xi^\prime} + \norm{u - u^\prime})
    ,\label{eq:proof_condition_1}
\end{align}
where in $(a)$ we solve the integrals, in $(b)$ we retain only the zeroth-order term in the Taylor expansion and bound the remainder linearly in $\pr$ in any interval $[0,\bprsi]$ for some $\bprsi > 0$, while in $(c)$ we use the definition of $\alpha_\prct$, the design choice $\prct = \ratio\pr$ and the fact that $\pr \in (0,\bprsi)$.
Hence, by specializing~\eqref{eq:proof_condition_1} with $u = \Pi(z)$ and $u^\prime = \Pi(z^\prime)$ and using the $\lip_\Pi$-Lipschitz continuity of $\Pi$ (cf. Assumption~\ref{ass:ges_MPC}), we conclude that condition~\eqref{eq:slow_condition} is satisfied with $\lip_{\slow} := \lip_{\slowc}\exp(\sqrt{\lip_{\slowc}^2\bar{\pr}_1^2 + \lip_{\fastc}^2/\ratio^2})\sqrt{1 + \lip_\Pi^2}$.

As regards condition~\eqref{eq:fast_condition}, we note that, by definition of $\fast$ (cf.~\eqref{eq:fast_definition}) and using the triangle inequality, it holds
\begin{align}
    &\norm{\fast(x,\xi,u,\pr) - \fast(x,\xi^\prime,u^\prime,\pr)}
    \norm{\xi - \xi^\prime} \label{eq:bound_for_fast_condition} 
    %
    + \frac{1}{\prct}\int_{0}^\pr \norm{\fastc(\xcont(\tint),\xicont(\tint),u) - \fastc(\xcont^\prime(\tint),\xicont^\prime(\tint),u^\prime)} d\tint.
\end{align}
Hence, by slightly adapting the arguments used in~\eqref{eq:third_bound},~\eqref{eq:fourth_bound}, and~\eqref{eq:proof_condition_1}, we can further bound~\eqref{eq:bound_for_fast_condition} as 
\begin{align}
    &%
    \norm{\fast(x,\xi,u,\pr) - \fast(x,\xi^\prime,u^\prime,\pr)}
    \leq
    (
        1 + \tfrac{\lip_{\fastc}}{\prct} \pr \exp(\alpha_\prct \pr)
        ) \norm{\xi - \xi^\prime}
        + \tfrac{\lip_{\fastc}}{\prct} \pr \exp(\alpha_\prct \pr) \norm{u - u^\prime}.
        \label{eq:proof_condition_2_pre}
    \end{align}
    By combining~\eqref{eq:proof_condition_2_pre} in the case with $u = \Pi(z)$ and $u^\prime = \Pi(z^\prime)$, the $\lip_\Pi$-Lipschitz continuity of $\Pi$ (cf. Assumption~\ref{ass:ges_MPC}), the $\lip_\cA$-Lipschitz continuity of $\cA$ (cf. Assumption~\ref{ass:ges_MPC}), the triangle inequality, and $\prct = \ratio \pr$, we get 
    \begin{align}
        &
        \norm{\begin{bmatrix} 
                    \fast(x,\xi,\Pi(z),\pr) - \fast(x,\xi^\prime,\Pi(z^\prime),\pr)
                    \\
                    \mpc(z,x) - \mpc(z^\prime,x)
                \end{bmatrix}}
                \notag\\
        &\leq
        (
            1 + \tfrac{\lip_{\fastc}}{\ratio} \exp(\sqrt{\lip_{\slowc}^2\bar{\pr}_1^2 + \lip_{\fastc}^2/\ratio^2})
            ) \norm{\xi - \xi^\prime}
            \label{eq:proof_condition_pre}
            + (
                \tfrac{\lip_{\fastc}\lip_{\Pi}}{\ratio} \exp(\sqrt{\lip_{\slowc}^2\bar{\pr}_1^2 + \lip_{\fastc}^2/\ratio^2}) + \lip_{\cA}
                )
                \norm{z - z^\prime}.
    \end{align}
    Hence, given $\lip_1 := 1 + \tfrac{\lip_{\fastc}}{\ratio} \exp(\sqrt{\lip_{\slowc}^2\bar{\pr}_1^2 + \lip_{\fastc}^2/\ratio^2})$ and $\lip_2 := \tfrac{\lip_{\fastc}\lip_{\Pi}}{\ratio} \exp(\sqrt{\lip_{\slowc}^2\bar{\pr}_1^2 + \lip_{\fastc}^2/\ratio^2}) + \lip_{\cA}$, condition~\eqref{eq:fast_condition} is satisfied with $\lip_{\fast} := \sqrt{\lip_1^2 + \lip_2^2}$.

To prove~\eqref{eq:red_condition}, we first write the continuous-time counterpart of system~\eqref{eq:redcuced_system}. 
Without loss of generality, we consider the time interval $[0,\pr)$ and this system reads as
\begin{align}\label{eq:reduced_system_ct}
    \dot{x}\ct(\tti) = \redc(x\ct(\tti),\ustar(x)),
\end{align}
with $x\ct(0) = x$ and $\ustar(x) := \Pi(\zstar(x))$.
Now, by definition of $\red$ (cf.~\eqref{eq:red}) and $\slow$ (cf.~\eqref{eq:slow_definition}), we note that
\begin{subequations}\label{eq:derivatives}
    \begin{align}
        &\nabla_3 \red(x,\ustar(x),\pr) = \redc(\xcont(\pr),\ustar(x)) 
        \\
        &\nabla_3^2 \red(x,\ustar(x),\pr) 
        = \nabla_1 \redc(\xcont(\pr),\ustar(x)) \redc(\xcont(\pr),\ustar(x)),
    \end{align}
\end{subequations}
where, with a slight abuse of notation, $\xcont(\pr)$ denotes the solution of system~\eqref{eq:reduced_system_ct} at time $\pr$ with initial condition $x$.
With these results at hand, we compute the Taylor's expansion of $\red(x,\ustar(x),\pr)$ around $(x,\ustar(x),0)$ with the integral form of the second-order remainder, namely
\begin{align}
    \red(x,\ustar(x),\pr) 
    &= \red(x,\ustar(x),0) + \delta \nabla_3\red(x,\ustar(x),0) 
    + \int_{0}^\pr (\pr - \tint)\nabla_3^2 \red(x,\ustar(x),\tint) d\tint
    \notag\\
    &\stackrel{(a)}{=}
    x + \delta \redc(x,\ustar(x)) 
    \label{eq:expansion_chi}
    + \int_{0}^\pr (\pr - \tint)\nabla_1 \redc(\xcont(\tint),\ustar(x))
    \redc(\xcont(\tint),\ustar(x)) d\tint,
\end{align}
where in $(a)$ we use the definition of $\red$ (cf.~\eqref{eq:red}) and $\slow$ (cf.~\eqref{eq:slow_definition}) as well as~\eqref{eq:derivatives}.
Let us provide a bound for the norm of the second term in the right-hand side of~\eqref{eq:expansion_chi}.
To this end, since $\xstar$ is an equilibrium point of system~\eqref{eq:reduced_system_ct} (cf. Assumption~\ref{ass:red_ges}), we can write 
\begin{align}
    \norm{\redc(x,\ustar(x))}
    &= \norm{\redc(x,\ustar(x)) - \redc(\xstar,\ustar(\xstar))} 
    \stackrel{(a)}{\leq} \lip_{\slowc}\sqrt{1 + \lipeq^2(1 +\lip_{\Pi}^2\lip_{\zstar}^2) + \lip_{\Pi}^2\lip_{\zstar}^2}\norm{x - \xstar},
    \label{eq:bound_linear_term}
\end{align}
where in $(a)$ we use the definition of $\redc$ (cf.~\eqref{eq:redc}) and the Lipschitz continuity of $\slowc$, $\xieq$, $\Pi$, and $\zstar$ (see Assumptions~\ref{ass:slowness},~\ref{ass:ges_fast}, and~\ref{ass:ges_MPC}).
Now, let us provide a bound for the norm of the integral remainder term.
Given $\lip_{\redc} := \lip_{\slowc}\sqrt{1 + \lipeq^2(1 + \lip_{\Pi}^2\lip_{\zstar}^2)}$, $\redc$ is $\lip_{\redc}$-Lipschitz continuous with respect to the first argument (implied by the definition of $\redc$ in~\eqref{eq:redc} and the Lipschitz continuity of $\slowc$, $\xieq$, $\Pi$, and $\zstar$, see Assumptions~\ref{ass:slowness},~\ref{ass:ges_fast}, and~\ref{ass:ges_MPC}, respectively).
By combining this fact with the Cauchy-Schwarz inequality and the fact that $\xstar$ is an equilibrium point of system~\eqref{eq:reduced_system_ct} (cf. Assumption~\ref{ass:red_ges}), we get
\begin{align}
    \int_{0}^\pr (\pr - \tint)\norm{\nabla_1 \redc(\xcont(\tint),\ustar(x))\redc(\xcont(\tint),\ustar(x))} d\tint
    &\leq 
    \int_{0}^\pr (\pr - \tint)\lip_{\redc}\norm{\redc(\xcont(\tint),\ustar(x)) -  \redc(\xstar,\ustar(\xstar))} d\tint
    \notag\\
    &\stackrel{(a)}{\leq} \int_{0}^\pr (\pr - \tint)\lip_{\redc}^2\norm{\xcont(\tint) - \xstar}d\tint
    \notag\\
    &\hspace{.4cm} 
    + \int_{0}^\pr (\pr - \tint)\lip_{\redc}\lip_{\slowc}\lip_{\Pi} \lip_{\zstar}\norm{\xstar - x} d\tint,
    \label{eq:bound}
\end{align}
where in $(a)$ we combine the definition of $\redc$ (cf.~\eqref{eq:redc}) with the Lipschitz continuity of $\slowc$, $\Pi$, and $\zstar$ (see Assumptions~\ref{ass:slowness},~\ref{ass:ges_fast}, and~\ref{ass:ges_MPC}, respectively).
    Now, let us focus on $\norm{\xcont(\tint) - \xstar}$.
    Let us introduce $\tilde{x}\ct := \xcont - \xstar$, which allows us to accordingly rewrite system~\eqref{eq:reduced_system_ct} as
    \begin{align}\label{eq:reduced_system_ct_error}
        \dot{\tilde{x}}\ct(t) = \redc(\tilde{x}\ct(t) + \xstar,\ustar(x)).
    \end{align}
    Then, for all $\tint \in \R$, it holds 
    \begin{align}
        \tfrac{d}{d \tint}\norm{\tilde{x}\ct(\tint)} &\leq \norm{\dot{\tilde{x}}\ct(\tint)}
        %
        %
        %
        %
        %
        %
        %
        %
        %
        \stackrel{(a)}{\leq} \lip_{\redc}\norm{\tilde{x}\ct(\tint)}  +  \lip_{\slowc}\lip_{\Pi}\lip_{\zstar}\norm{\xstar  -  x},
    \end{align}
    where in $(a)$ we add $\redc(x_{\star},\ustar(\xstar)) = 0$ and use the Lipschitz continuity of $\redc$, $\slowc$, $\mpc$ and $\zstar$.
    Then, by Gronwall Lemma (see, e.g., \cite[Th. 1.3.1]{ames1997inequalities}), it holds
    \begin{align}
        \norm{\tilde{x}\ct(\tint)} &\leq \exp(\lip_{\redc} \tint)\norm{x - \xstar}
         + \lip_{\slowc}\lip_{\Pi}\lip_{\zstar}\norm{\x  - \xstar}\int_{0}^\tint \exp(\lip_{\redc}(\tint  -  \sigma)) d\sigma,
                \label{eq:gronwall_red}\\
        %
        &=(
            \exp(\lip_{\redc}\tint)  +  \tfrac{\lip_{\slowc}\lip_{\Pi} \lip_{\zstar}}{\lip_{\redc}} (\exp(\lip_{\redc}\tint) - 1)
        )\norm{x - \xstar}.\notag
    \end{align}
    By putting \eqref{eq:gronwall_red} into \eqref{eq:bound}, we get
    \begin{align}
        &\int_{0}^\pr (\pr - \tint)\norm{\nabla_1 \redc(\xcont(\tint),\ustar(x))\redc(\xcont(\tint),\ustar(x))} d\tint
        \notag\\
        &\leq
         (1  +  \lip_{\slowc}\lip_{\Pi} \lip_{\zstar} \lip_{\redc})\frac{\exp(\lip_{\redc} \pr)  -  \lip_{\redc} \pr  -  1}{\lip_{\redc}^2}\norm{x  -  \xstar}
        + \lip_{\redc}\lip_{\slowc}\lip_{\Pi} \lip_{\zstar} \tfrac{\pr^2}{2}\norm{x - \xstar}
        \notag\\
        &\stackrel{(b)}{\leq}
        (1 + \lip_{\slowc}\lip_{\Pi} \lip_{\zstar} \lip_{\redc})\exp(\lip_{\redc} \bprsi)\tfrac{\pr^2}{2}\norm{x - \xstar}
        + \lip_{\redc}\lip_{\slowc}\lip_{\Pi} \lip_{\zstar} \tfrac{\pr^2}{2}\norm{x - \xstar}
        ,
        \label{eq:second_order_term}
    \end{align}
    where %
    $(a)$ follows by a Taylor expansion of the exponential term up to the first-order term.

    Given $\tilde{\lip}_{\slow} > \lip_{\slowc}\sqrt{1 + \lipeq^2(1 +\lip_{\Pi}^2\lip_{\zstar}^2) + \lip_{\Pi}^2\lip_{\zstar}^2}$, the proof of~\eqref{eq:red_condition} follows by combining~\eqref{eq:expansion_chi},~\eqref{eq:bound_linear_term}, and~\eqref{eq:second_order_term} and by defining $\bar{\pr}_1 > 0$ as the largest value of $\pr$ such that 
    \begin{align*}
        &
        \pr\tilde{\lip}_{\slow} \geq \pr\lip_{\slowc}\sqrt{1 + \lipeq^2(1 +\lip_{\Pi}^2\lip_{\zstar}^2) + \lip_{\Pi}^2\lip_{\zstar}^2}
        %
        +
        \tfrac{\pr^2}{2}((1 + \lip_{\slowc}\lip_{\Pi} \lip_{\zstar} \lip_{\redc})\exp(\lip_{\redc} \pr) + \lip_{\redc}\lip_{\slowc}\lip_{\Pi} \lip_{\zstar}).
    \end{align*}

\subsection{Proof of Lemma~\ref{lemma:bl}}
\label{sec:bl_proof}

The proof consists in combining the stability properties of the continuous-time extra dynamics~\eqref{eq:fast_plant} (cf. Assumption~\ref{ass:ges_fast}) with those of the optimizer~\eqref{eq:mpc_dyn} (cf. Assumption~\ref{ass:ges_MPC}), and in suitably handling the drift of the fast plant equilibrium $\xieq(x, \Pi(\tzt + \zstar(x)))$ due to the variations of $\tzt$.
To this end, we consider the candidate Lyapunov function $U$ defined in~\eqref{eq:cU_def}.
We note that~\eqref{eq:cU_1} directly follow by~\eqref{eq:U_1} and~\eqref{eq:cL_1}.
To prove~\eqref{eq:cU_2}, we first define %
    \begin{align}
        \txi\ud{+}\du{\text{\tiny nom}} &:= \tfast(x, \txi, \tz, \pr).
        \label{eq:txi_plus_nom}
    \end{align}
Then, we consider the increment $\cG(\txi\ud{+}\du{\text{\tiny nom}}) - \cG(\txi)$ of $\cG$ (cf. Assumption~\ref{ass:ges_fast}) along the trajectories of the nominal fast dynamics $\txi\ud{+}\du{\text{\tiny nom}} = \tfast(x, \txi, \tz, \pr)$.
To characterize $\cG(\txi\ud{+}\du{\text{\tiny nom}})$, we note that its continuous-time dynamics reads as
\begin{align}
    \dot{\txi}(\tti) = \frac{1}{\prct}\fastc(\x, \txi(\tti) + \xieq(x,\Pi(\tz + \zstar(x))), \Pi(\tz + \zstar(x))).
\end{align}
Hence, by invoking~\eqref{eq:U_2} in Assumption~\ref{ass:ges_fast}, the time-derivative of the Lyapunov function $\cG(\txi)$ satisfies
\begin{align}
    \frac{d}{d \tti} \cG(\txi(\tti)) %
    &\leq - \tfrac{a_3}{\prct} \|\txi(\tti)\|^2%
    \stackrel{(a)}{\leq} - \tfrac{a_3}{\prct a_2}\cG(\txi(\tti))\label{eq:dot_cG_2},
\end{align}
where in $(a)$ we apply~\eqref{eq:U_1} (cf. Assumption~\ref{ass:ges_fast}).
By multiplying~\eqref{eq:dot_cG_2} by $\exp(\tfrac{a_3}{\prct a_2} \tint)$, we rewrite~\eqref{eq:dot_cG_2} as
\begin{align}
    \frac{d}{d \tti} (
        \cG(\txi(\tti)) \exp(\tfrac{a_3}{\prct a_2} \tti)) \leq 0.
    \label{eq:dot_cG_3}
\end{align}
Hence, integrating both sides of \eqref{eq:dot_cG_3} over the sampling period from $\tau=0$ to $\tau=\pr$ yields
\begin{align}
    \int_0^{\pr} \frac{d}{d \tint} (
        \cG(\txi(\tint)) \exp(\tfrac{a_3}{\prct a_2} \tint)
    ) d\tint \leq 0.
    \label{eq:integral_cG}
\end{align}
Then, by using the Fundamental Theorem of Calculus and since $\txi(\pr) = \txi\ud{+}\du{\text{\tiny nom}}$, we rewrite~\eqref{eq:integral_cG} as
\begin{align}
    \int_0^{\pr} \frac{d}{d \tint}(
        \cG(\txi(\tint)) \exp(\tfrac{a_3}{\prct a_2} \tint)
    ) d\tint
    &%
    =  \cG(\txi\ud{+}\du{\text{\tiny nom}}) \exp(\tfrac{a_3}{\prct a_2} \pr) - \cG(\txi)
    = \exp(\tfrac{a_3}{\prct a_2} \pr) \cG(\txi\ud{+}\du{\text{\tiny nom}})  -  \cG(\txi).
    \label{eq:FTC_cG}
\end{align}
We plug~\eqref{eq:FTC_cG} in~\eqref{eq:integral_cG} and add $\pm\exp(\tfrac{a_3}{\prct a_2} \pr) \cG(\txi)$ to get
\begin{align}
    \cG(\txi\ud{+}\du{\text{\tiny nom}}) - \cG(\txi) \leq -(1 - \exp(-\tfrac{a_3}{\prct a_2} \pr)) \cG(\txi).
     \label{eq:FTC_cG_2}
\end{align}
Now, we use the design choice $\prct = \ratio\pr$ to rewrite~\eqref{eq:FTC_cG_2} as
\begin{align}
    \cG(\txi\ud{+}\du{\text{\tiny nom}}) - \cG(\txi) \leq -(1 - \exp(-\tfrac{a_3}{\ratio a_2})) \cG(\txi). 
    \label{eq:Delta_cG_nom_0}
\end{align}
Then, given $a_5 := 1 - \exp(- \tfrac{a_3}{\ratio a_2}) > 0$, we bound~\eqref{eq:Delta_cG_nom_0} as
\begin{align}
    \cG(\txi\ud{+}\du{\text{\tiny nom}}) - \cG(\txi) &\leq -a_5 \cG(\txi)
    \stackrel{(a)}{\leq} -\tilde{a}_3 \|\txi\|^2
    ,\label{eq:Delta_cG_nom}
\end{align}
where in $(a)$ we apply the lower bound in~\eqref{eq:U_1} (cf. Assumption~\ref{ass:ges_fast}) and define $\tilde{a}_3 := a_5 a_1$.
With this result at hand, we now consider the increment $\Delta U(\txi,\tz) := U(\txi\ud{+}, \tz\ud{+}) - U(\txi, \tz)$ of $U$ (cf.~\eqref{eq:cU_def}) along the trajectories of~\eqref{eq:BL} and, by using~\eqref{eq:cL_2} (cf. Assumption~\ref{ass:ges_MPC}), we get
\begin{align}
    \Delta U(\txi, \tz) 
    &\leq
    \cG(\txi\ud{+}) 
    - \cG(\txi) -\kappa d_3 \norm{\tz}^2
    \notag 
    \\
     &\stackrel{(a)}{=} 
    \cG(\txi\ud{+}) 
    - \cG(\txi\ud{+}\du{\text{\tiny nom}})
     + 
    \cG(\txi\ud{+}\du{\text{\tiny nom}})  -  \cG(\txi)
     -   \kappa d_3 \norm{\tz}^2
    \notag 
    \\
     &\stackrel{(b)}{\leq} 
     \cG(\txi\ud{+})
      -   
    \cG(\txi\ud{+}\du{\text{\tiny nom}}) - \tilde{a}_3 \|\txi\|^2
      -  \kappa d_3 \norm{\tz}^2,
    \label{eq:Delta_cU}
\end{align}
where in $(a)$ we add $\pm \cG(\txi\ud{+}\du{\text{\tiny nom}})$, while in $(b)$ we use~\eqref{eq:Delta_cG_nom}.
We now focus on $\cG(\txi\ud{+}) - \cG(\txi\ud{+}\du{\text{\tiny nom}})$.
By using the Lipschitz continuity of $\nabla \cG$ (cf. Assumption~\ref{ass:ges_fast}) and the definitions of $\txi\ud{+}$ (cf.~\eqref{eq:txi_plus}) and $\txi\ud{+}\du{\text{\tiny nom}}$ (cf.~\eqref{eq:txi_plus_nom}), we get
\begin{align}
    \cG(\txi\ud{+})  -  \cG(\txi\ud{+}\du{\text{\tiny nom}})
    &\leq  
    \tfrac{a_4}{2} \norm{\Delta \xieq(x, x, \tz,\tz\ud{+})}^2\label{eq:pert_bound}
    +  a_4  \norm{\Delta \xieq(x, x, \tz,\tz\ud{+})}\|\tfast(x, \txi, \tz, \pr)\|.
\end{align}
    By using the definition of $\tfast$ (cf.~\eqref{eq:tfast}), the fact that $\xieq(x,\Pi(\tz + \zstar(x))) = \fast(x,\xieq(x,\Pi(\tz + \zstar(x))), \Pi(\tz + \zstar(x)), \pr)$ (see the definition of $\fast$ in~\eqref{eq:fast_definition} and the equilibrium condition~\eqref{eq:xieq} in Assumption~\ref{ass:ges_fast}), and the fact that $\fast$ is $\lip_{\fast}$-Lipschitz continuous with $\lip_{\fast} := (1 + \frac{\lip_{\fastc}}{\ratio}\exp(\max\{\lip_{\slowc}\bar{\pr}_2,\frac{\lip_{\fastc}}{\ratio}\}))$ (see~\eqref{eq:proof_condition_2_pre} and~\eqref{eq:proof_condition_pre} in the proof of Lemma~\ref{lemma:single-int} in Appendix~\ref{sec:proof_lemma_single_int}), we get
    \begin{align}
        &
        \|\tfast(x, \txi, \tz, \pr)\|
        \leq
        \lip_{\fast}\|\txi\|,\label{eq:bound_norm_eta}
\end{align}
By observing the definition of $\Delta \xieq$ (cf.~\eqref{eq:delta_xieq}) and using the Lipschitz continuity of $\xieq$ (cf. Assumption~\ref{ass:ges_fast}), we get
\begin{align}
    \norm{\Delta \xieq(x, x,\tz,\tz\ud{+})} &\leq L_{\xi} \norm{\Pi(\tz\ud{+}  +  \zstar(x))  -  \Pi(\tz + \zstar(x))}
    %
    %
    \stackrel{(a)}{\leq} L_{\xi}\lip_{\Pi}(L_{\mpc} + 1)\norm{\tz},\label{eq:bound_norm_error}
\end{align}
where in $(a)$ we use the Cauchy-Schwarz inequality, 
the definition of $\tz\ud{+}$ (cf.~\eqref{eq:tz_plus}), and the triangle inequality, the fact that $\tmpc(0,x) = 0$ for all $x \in \feas$ (see~\eqref{eq:optimum_equilibrium} in Assumption~\ref{ass:ges_MPC}), and the Lipschitz continuity of $\tmpc$ and $\Pi$ (cf. Assumption~\ref{ass:ges_MPC}).
Therefore, by using~\eqref{eq:bound_norm_eta} and~\eqref{eq:bound_norm_error}, we further bound the right-hand side of~\eqref{eq:pert_bound} as
\begin{align}
   & \cG(\txi\ud{+}) - \cG(\txi\ud{+}\du{\text{\tiny nom}}) \leq 2k_1 \|\txi\|\norm{\tz} + k_2 \norm{\tz}^2, 
    \label{eq:drift_term}  
\end{align}
in which $k_1 := a_4 L_{\fast} L_{\xi} L_{\Pi} (L_{\mpc} + 1)/2$ and $k_2 := \frac{a_4}{2} L_{\xi}^2 \lip_{\Pi}^2(L_{\mpc} + 1)^2$.
By using~\eqref{eq:drift_term}, we bound~\eqref{eq:Delta_cU} as
\begin{align}
    \Delta U(\txi, \tz) &\leq -\tilde{a}_3 \|\txi\|^2 - (\kappa d_3 - k_2) \norm{\tz}^2 + 2 k_1 \|\txi\|\norm{\tz}
    \stackrel{(a)}{=}
    -  
    \begin{bmatrix}
        \norm{\txi}
        \\
        \norm{\tz}
    \end{bmatrix}\T
    \underbrace{
    \begin{bmatrix}
        \tilde{a}_3 & - k_1
        \\
        - k_1 & \kappa d_3 - k_2
    \end{bmatrix}}_{H(\kappa)} 
    \begin{bmatrix}
        \norm{\txi}
        \\
        \norm{\tz}
    \end{bmatrix},
\end{align}
where in $(a)$ we adopt a compact matrix notation.
By Sylvester criterion, we have
\begin{align*}
   H(\kappa) \succ 0 &\iff \tilde{a}_3(\kappa d_3 - k_2) > k_1^2
   \iff \kappa > \tfrac{k_1^2}{\tilde{a}_3 d_3} + \tfrac{k_2}{d_3}.
\end{align*}
The proof follows by setting $\bar{\kappa} := \frac{k_1^2}{\tilde{a}_3 d_3} + \frac{k_2}{d_3}$ and denoting with $b_3 > 0$ the smallest eigenvalue of $H(\kappa)$.

\subsection{Proof of Lemma~\ref{lem:rs}}
\label{sec:proof_lemma_rs}

        For later use, let us define %
        \begin{subequations}\label{eq:definitions}
            \begin{align}
                c_B &:= \lip_{\slowc}\lip_{\Pi}{\lip_{\zstar}}\label{eq:c_B}
                \\
                \lip_{\redc} &:= \lip_{\slowc}\sqrt{1 + \lipeq^2(1 + \lip_{\Pi}^2)}
                \label{eq:lip_redc}
                \\
                K_1(\pr) &:= (1 + c_B \pr) \exp(\lip_{\redc} \pr)
                \label{eq:K_1}
                \\
                K_2(\pr) &:= \lip_{\redc} K_1(\pr) + c_B.
                \label{eq:K_2}
            \end{align}
        \end{subequations}
        By definition of $\gammam$ (cf.~\eqref{eq:gammam}), we recall that as long as $x \in \Omega_W(\gammam)$ it holds $x \in \feas$ too and, thus, we can enjoy the stability property~\eqref{eq:descent_property_mpc} (cf. Assumption~\ref{ass:red_ges}).
        With this in mind, let us arbitrarily choose $\tilde{c}_3 \in (0, c_3)$ and define $\bar{\pr}_5 :=  \frac{1}{2 \lip_{f} + \lip_{f} \lip_{\Pi} \lip_{\zstar}}$ and, by using the definitions in~\eqref{eq:definitions}, $\bar{\pr}_6$ as the largest value of $\pr$ such that
        \begin{align}\label{eq:bar_pr_3}
            c_3 -\tilde{c}_3  > \pr\left(
                        c_3 K_2(\pr) - c_3 \tfrac{\pr}{3} K_2(\pr)^2 + \tfrac{c_4 c_B}{2} K_1(\pr) K_2(\pr)
                    \right).
        \end{align}
        Then, since $x \in \Omega_W(\gamma)$ with $\gamma \in (0, \gammam)$ by hypothesis, let us prove that $x(\pr) \in \Omega_W(\gammam)$ for all $\pr \in (0, \bar{\pr}_3)$, where $\bar{\pr}_3 := \min\{\bar{\pr}_5,\bar{\pr}_6\}$.
        We proceed by contradiction and, thus, suppose that there exists an exit time instant $\bar{\tti} \in (0, \pr]$ such that $W(x(\bar{\tti})) = \gammam$.
        Hence, $x(\tti) \in \Omega_W(\gammam)$ for all $\tti \in [0, \bar{\tti}]$ and we can write
        \begin{align}
            W(x(\bar{\tti})) 
            &= W(x) + \int_0^{\bar{\tti}} \nabla W(x(\tint))\T \redc(x(\tint), \Pi(\zstar(x))) d\tint
            \notag \\
            &\stackrel{(a)}{\leq} W(x) - c_3 \int_0^{\bar{\tti}} \norm{x(\tint) - \xstar}^2 d\tint
            + c_4 c_B \int_0^{\bar{\tti}} \norm{x(\tint) - \xstar}\norm{x(\tint) - x} d\tint,\label{eq:W_bar_tti}
        \end{align}
        where in $(a)$ we add and subtract $\int_0^{\bar{\tti}} \nabla W(x(\tint)){^\top} \redc(x(\tint), \Pi(\zstar(x(\tint)))) d\tint$,
        use~\eqref{eq:descent_property_mpc} in Assumption~\ref{ass:red_ges} and the Lipschitz continuity of $\nabla W$ (cf. Assumption~\ref{ass:red_ges}), $\redc$ (cf. Assumption~\ref{ass:slowness} and~\ref{ass:ges_fast}), $\Pi$, and $\zstar$ (cf. Assumption~\ref{ass:ges_MPC}).
        Now we focus on $\norm{x(\tti) - \xstar}$.
        By using the triangle inequality and since $\redc(\xstar, \Pi(\zstar(\xstar))) = 0$ (cf. Assumption~\ref{ass:red_ges}), we get
        \begin{align}
            \norm{x(\tti) - \xstar}
            &\leq
            \norm{x - \xstar} 
            + \int_{0}^{\tti} \norm{\redc(x(\tint), \Pi(\zstar(x))) - \redc(\xstar, \Pi(\zstar(\xstar)))} d\tint
            \notag \\
            &\stackrel{(a)}{\leq} 
             \norm{x - \xstar}(1 + c_B \tti)  +  \lip_{\redc} \int_{0}^{\tti} \norm{x(\tint) - \xstar} d\tint,
            \label{eq:bound_x_tau}
        \end{align}
        where in $(a)$ we add and subtract $\redc(\xstar, \Pi(\zstar(x(\tint))))$ inside the norm, use the triangle inequality, the Lipschitz continuity of $\slow$ (cf. Assumption~\ref{ass:slowness}), $\xieq$ (cf. Assumption~\ref{ass:ges_fast}), $\Pi$, and $\zstar$ (cf. Assumption~\ref{ass:ges_MPC}), and the constant $\lip_{\redc}$ (cf.~\eqref{eq:lip_redc}).
        Now, by using the Gronwall inequality (see, e.g., \cite[Th. 1.3.1]{ames1997inequalities}), we bound~\eqref{eq:bound_x_tau} as
        \begin{align}
            \norm{\x(\tti) - \xstar} &\leq \norm{x - \xstar} (1 + c_B \tti) \exp(\lip_{\redc} \tti)
            \stackrel{(a)}{\leq} 
            \norm{x - \xstar} K_1(\pr),
            \label{eq:boundAA}
        \end{align}
        where in $(a)$ we use the fact that $\tti \leq \pr$ and $K_1$ (cf.~\eqref{eq:K_1}).
        As for the term $\norm{x(\tti) - x}$, we write
        \begin{align}
            \norm{x(\tti) - x} 
            = 
            \int_{0}^{\tti} \norm{\redc(x(\tint), \Pi(\zstar(x)))} d\tint
            &\stackrel{(a)}{\leq} \int_{0}^{\tti} (\lip_{\redc} \norm{x(\tint) - \xstar} + c_B \norm{x - \xstar}) d\tint
            \notag \\
            &\stackrel{(b)}{\leq}  (\lip_{\redc} K_1(\pr) + c_B) \norm{x - \xstar} \tti
            \notag \\
            &\stackrel{(c)}{=} \tti K_2(\pr) \norm{x - \xstar},
            \label{eq:boundBB}
        \end{align}
        where in $(a)$ we use the same arguments used in~\eqref{eq:bound_x_tau} to handle the norm inside the integral, in $(b)$ we apply the bound~\eqref{eq:boundAA} on $\norm{x(s) - \xstar}$, the fact that $K_1(\pr)$ is increasing, and that $\tti \leq \pr$, and solve the integral, while in $(c)$ we use $K_2(\pr)$ (cf.~\eqref{eq:K_2}).
        We now go back to the bound on $\norm{x(\tti) - \xstar}$ and by adding and subtracting $x$ inside the norm, and using the triangle inequality, we get
        \begin{align}
            \norm{x(\tti) - \xstar}
            &\geq \norm{x - \xstar} - \norm{x(\tti) - x} 
            %
            %
            \stackrel{(a)}{\geq} \norm{x - \xstar} (1 - \tti K_2(\pr)),\label{eq:norm}
        \end{align}
        where in $(a)$ we use~\eqref{eq:boundBB}.
        Since $\tti \leq \pr$ and $\pr \in (0,\bar{\pr}_3)$, it holds $1 - \tti K_2(\pr) > 0$ and, thus, the bound~\eqref{eq:norm} implies
        \begin{align}
            \norm{x(\tti) - \xstar}^2 \geq \norm{x - \xstar}^2 (1 - \tti K_2(\pr))^2. \label{eq:boundCC}
        \end{align}
        Hence, by using~\eqref{eq:W_bar_tti},~\eqref{eq:boundAA},~\eqref{eq:boundBB}, and~\eqref{eq:boundCC}, we get
        \begin{align}
            W(x(\bar{\tti})) - W(x)
            %
            %
            %
            %
            %
            &\leq -c_3 \int_0^{\bar{\tti}} (1 - \tint K_2(\bar{\tti}))^2 d\tint \norm{x - \xstar}^2 
            + c_4 c_B \int_0^{\bar{\tti}} K_1(\bar{\tti}) K_2(\bar{\tti}) \tint d\tint \norm{x - \xstar}^2,
            \notag\\
            &\stackrel{(a)}{=}
            - \bar{\tti}c_3 \norm{x - \xstar}^2
             +  \bar{\tti}^2\left(c_3 K_2(\bar{\tti}) - c_3 \tfrac{\bar{\tti}}{3} K_2(\bar{\tti})^2\right)
            \norm{x - \xstar}^2
            + \bar{\tti}^2\frac{c_4 c_B}{2} K_1(\bar{\tti}) K_2(\bar{\tti})\norm{x - \xstar}^2
            \notag\\
            &\stackrel{(b)}{\leq}
            - \bar{\tti}\tilde{c}_3\norm{x - \xstar}^2,
        \end{align}
        where in $(a)$ we solve the two integrals, while in $(b)$ we use $\bar{\tti} \in (0,\pr]$, $\pr \in (0, \bar{\pr}_3)$, $\bar{\pr}_3 = \min\{\bar{\pr}_5, \bar{\pr}_6\}$, and the definition of $\bar{\pr}_5$ (cf.~\eqref{eq:bar_pr_3}).
        Hence, we have $W(x(\bar{\tti})) \leq W(x) \leq \gamma$, which contradicts $W(x(\bar{\tti})) = \gammam$, thus proving that $x(\tti) \in \Omega_W(\gammam) \subseteq \feas$ for all $\tti \in [0, \pr)$.
        The proof then follows by repeating the above steps in the entire interval $[0,\pr]$.

%
%
%
%
%
%
%
%
%
%
%
%
%
%
%
%
%
%
%
%
%
%
%
%
%
%

    %

    %
    %
    %
    %
    %
    %
    %
    %
    %
    %
    %
    %
    %
    %
    %
    %
    %
    %
    %
    %

    %
    %
    %
    %
    %
    %
    %
    %
    %
    %
    %
    %
    %
    %
    %
    %
    %
    %
    %

    %

    %
    %

    %

    %
    
    %
    %
    %
    %
    %
    %

    %
    
    %
    %
    %
    %
    %
    %
    %
    %
    %
    %
    
    %
    %
    %
    %
    %
    %
    %
    %
    %
    %
    %
    %
    %
    %
    %
    %
    %
    %
    %
    %
    %
    %
    %
    %
    %
    %
    %
    %
    %
    %
    %
    %
    %
    %
    %
    %
    %
    %
    %
    %
    %
    %
    %
    %
    %
    %
    %
    %
    %
    %
    %
    %
    %
    %
    %
    %
    %
    %
    %
    %
    %
    %
    %
    %
    %
    %
    %
    %
    %
    %
    %
    %
    %
    %
    %
     
    %
    %
    %
    %
    %
    %
    %
    %
    %
    %
    %
    %
    %
    %
    %
    %
    %
    %
    %
    %
    %
    %
    %
    %
    %
    %
    %
    %
    %
    %
    %
    %
    %
    %
    %
    %
    %
    %
    %
    %
    %
    %
    %
    %
    %
    %
    %
    %
    %
    %
    %
    %
    %
    %
    %
    %
    %
    %
    %
    %
    %
    %
    %
    %
    %
    %
    %
    %
    %
    %
    %
    %
    %
    %
    %
    %
    %
    %
    %
    %
    %
    %
    %
    %
    %
    %
    %
    %
    %
    %
    %
    %
    %
    %
    %
    %
    %
    %
    %
    %
    %
    %
    %
    %
    %
    %

    %
    %
    %

    %

    %
    %
    %
    %
    %
    %
    %
    %
    %
    %
    %
    %
    %
    %
    %
    %
    %
    %
    %
    %
    %
    %
    %
    %
    %
    %
    %
    %
    %
    %
    %
    %
    %
    %
    %
    %
    %
    %
    %
    %
    %
    %
    %
    %
    %
    %
    %
    %
    %

    %

    %
    %
    %
    %
    %
    %
    %
    %
    %
    %
    %
    %
    %
    %
    %
    %
    %
    %
    %
    %
    %
    %
    %
    %
    %
    %
    %
    %
    %
    %
    %
    %
    %
    %
    %
    %
    %
    %
    %
    %
    %
    %
    %
    %
    %
    %
    %
    %
    %
    %
    %
    %
    %
    %
    %
    %
    %
    %
    %
    %
    %
    %
    %
    %
    %
    %
    %
    %

%
%
%
%
%
%
%
%
%
%
%
%
%
%
%
%
%
%
%
%
%
%
%
%
%
%
%
%
%
%
%
%
%
%
%
%
%
%
%
%
%
%
%
%
%
%
%
%
%
%
%
%
%
%
%
%
%
%
%
%
%
%
%
%
%
%
%
%
%
%
%
%
%
%
%
%
%
%
%
%
%
%
%
%
%
%
%
%
%
%
%
%
%
%
%
%
%
%
%
%
%
%
%
%
%
%
%
%
%
%
%
%
%
%
%
%
%
%
%
%
%
%
%
%
%
%
%
%
%
%
%
%
%
%
%
%
%
%
%
%
%
%

%

%
%
%
%
%
%
%
%
%
%
%
%
%
%
%
%
%
%
%
%
%
%
%
%
%
%
%
%
%
%
%
%
%
%

%
%
%
%
%
%
%
%
%
%
%
%
%
%
%
%
%
%
%
%
%
%
%
%
%
%
%
%
%
%
%
%
%
%
%
%
%
%
%
%
%
%
%
%
%
%
%
%
%
%
%
%
%
%
%
%


%
%
%
%
%
%
%
%
%
%
%
%
%
%
%
%
%
%
%
%
%
%

\end{document}